\documentclass[12pt]{amsart}
 \usepackage{geometry}                
\usepackage{graphicx}
\usepackage{amssymb}
\usepackage{epstopdf}
\usepackage{color}
\usepackage{bbm, dsfont}
\usepackage{hyperref}
\usepackage[utf8]{inputenc}
\usepackage{amssymb,amsthm,amsmath,amssymb,wrapfig,dsfont}
\usepackage[dvipsnames]{xcolor}
\definecolor{myred}{RGB}{251,154,133}
\definecolor{myblue}{RGB}{153,206,227}
\definecolor{mylightblue}{RGB}{0, 150, 255}
\definecolor{mygreen}{RGB}{32, 210, 64}
\definecolor{mygray}{RGB}{220, 220, 220}

\usepackage{tikz}
\usetikzlibrary{decorations.pathmorphing}
\tikzset{snake it/.style={decorate, decoration=snake}}
\usetikzlibrary{shapes.geometric,positioning,decorations.pathreplacing}

\newtheorem{theorem}{Theorem}[section]
\newtheorem{lemma}[theorem]{Lemma}

\newtheorem{conjecture}[theorem]{Conjecture}
\theoremstyle{definition}
\newtheorem{definition}[theorem]{Definition}
\theoremstyle{remark}
\newtheorem{remark}[theorem]{Remark}

\theoremstyle{example}

\DeclareFontFamily{OML}{rsfs}{\skewchar\font'177}
\DeclareFontShape{OML}{rsfs}{m}{n}{ <5> <6> rsfs5 <7> <8> <9>
rsfs7 <10> <10.95> <12> <14.4> <17.28> <20.74> <24.88> rsfs10 }{}
\DeclareMathAlphabet{\mathfs}{OML}{rsfs}{m}{n}

\newcommand{\BZ}{{\mathbb{Z}}}

\newcommand{\CP}{{\mathcal{P}}}

\newcommand{\ind}{{\mathbbm{1}}}

\newcommand{\bae}{\begin{equation}\begin{aligned}}
\newcommand{\eae}{\end{aligned}\end{equation}}

\newtheorem{Proposition}[theorem]{Proposition}

\def\beq{ \begin{equation} }
\def\eeq{ \end{equation} }

\def\square{\vcenter{\vbox{\hrule height .4pt
  \hbox{\vrule width .4pt height 5pt \kern 5pt
        \vrule width .4pt} \hrule height .4pt}}}

\def\RR{\mathbb{R}}
\def\ZZ{\mathbb{Z}}

\def\var{\hbox{var}\,}

\begin{document}

\numberwithin{equation}{section} 

\title{On Covering Monotonic Paths with Simple Random Walk}

\author{Eviatar B. Procaccia}
\address[Eviatar B. Procaccia\footnote{Research supported by NSF grant DMS-1407558}]{Texas A\&M University}
\urladdr{www.math.tamu.edu/~procaccia}
\email{eviatarp@gmail.com}
 
\author{Yuan Zhang}
\address[Yuan Zhang]{Texas A\&M University}
\urladdr{http://www.math.tamu.edu/~yzhang1988/}
\email{yzhang1988@math.tamu.edu}


\maketitle
%
%
%

\begin{abstract}
In this paper we study the probability that a $d$ dimensional simple random walk (or the first $L$ steps of it) covers each point in a nearest neighbor path connecting 0 and the boundary of an $L_1$ ball. We show that among all such paths, the one that maximizes the covering probability is the monotonic increasing  one that stays within distance 1 from the diagonal. As a result, we can obtain an exponential upper bound on the decaying rate of covering probability of any such path when $d\ge 4$. 
\end{abstract}

\tableofcontents
\section{Introduction}
In this paper, we study the probability that a finite subset, especially the trace of a nearest neighbor path in $\ZZ^d$ is completely covered by the trace of a  $d$ dimensional simple random walk. 

\

For any finite subset $A\subset \ZZ^d$ and a $d$ dimensional simple random walk $\{X_n\}_{n=0}^\infty$ starting at 0, we say that $A$ is completely covered by the first $L$ steps of the random walk if 
$$
A\subseteq {\rm Trace}(X_0,X_1,\cdots, X_L):=\{x\in\BZ^d:\exists 0\le i\le L, X_i=x\}.
$$
For simplicity we state our first result for $d=2$. For integer $l_0\ge 0$ and the subspace of reflection $l: x=y+l_0$, we can define $\varphi_l: \ZZ^2\to\ZZ^2$ as the reflection mapping around $l$. I.e., for any $(x,y)\in \ZZ^2$,
$$
\varphi_l(x,y)=(l_0+y,x-l_0).
$$
Suppose two disjoint finite sets $A_0,B_0\subset \ZZ^2\cap \{x,y: x\le y+l_0\}$ both stay on the left of $l$. We then have the following theorem which states that the covering probability cannot get larger when we reflect one of them to the other side of the line:
\begin{theorem}
\label{theorem 2D random walk reflection}
For any integer $L\ge 0$,
$$
P\left(A_0\cup B_0\subseteq{\rm Trace}\big(\{X_n\}_{n=0}^L\big) \right)\ge P\left(A_0\cup \varphi_l(B_0)\subseteq{\rm Trace}\big(\{X_n\}_{n=0}^L\big) \right).
$$
\end{theorem}
\begin{remark}
By taking the union over all the $L$'s, one can immediately see the theorem also holds for $L=\infty$. \end{remark}

\begin{remark}
One would think (like the authors first did) that Theorem \ref{theorem 2D random walk reflection} should follow from repeated use of the reflection principle. Two problems arise when one explores this idea. The first is that reflecting a path does not conserve the hitting order within the sets, which makes it hard to determine the times of reflection. The second is that even if we consider the sets before and after reflection with the same hitting order we can get a contradiction to the monotonicity of cover probabilities with the specified order. See Figure \ref{fig:notmono} for an example. Here after hitting the reflected first vertex the path covers, the vertices we would like to hit in order 3, 4 and 5 block the way to the vertex we would like to visit second.
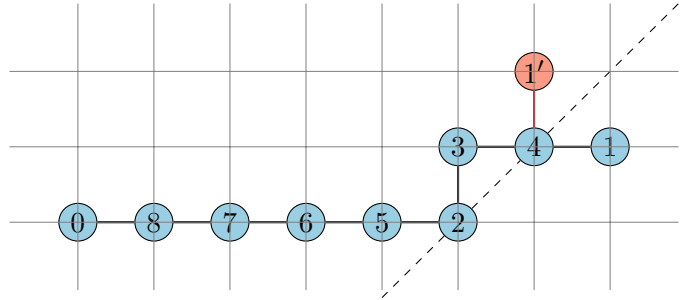
\begin{figure}[h!]
\centering
\begin{tikzpicture}
\tikzstyle{redcirc}=[circle,
draw=black,fill=myred,thin,inner sep=0pt,minimum size=5mm]
\tikzstyle{bluecirc}=[circle,
draw=black,fill=myblue,thin,inner sep=0pt,minimum size=5mm]

\node (v1) at (0,0) [bluecirc] {\small{$0$}};
\node (v2) at (1,0) [bluecirc] {\small{$8$}};
\node (v3) at (2,0) [bluecirc] {\small{$7$}};
\node (v4) at (3,0) [bluecirc] {\small{$6$}};
\node (v5) at (4,0) [bluecirc]{\small{$5$}};
\node (v6) at (5,0) [bluecirc] {\small{$2$}};
\node (v7) at (5,1) [bluecirc] {\small{$3$}};
\node (v8) at (6,1) [bluecirc] {\small{$4$}};
\node (v9) at (7,1) [bluecirc] {\small{$1$}};
\node (v10) at (6,2) [redcirc] {\small{$1'$}};

\draw [thick] (v1) to (v2);
\draw [thick] (v2) to (v3);
\draw [thick] (v3) to (v4);
\draw [thick] (v4) to (v5);
\draw [thick] (v5) to (v6);
\draw [thick] (v6) to (v7);
\draw [thick] (v7) to (v8);
\draw [thick] (v8) to (v9);
\draw [red,thick] (v8) to (v10);
\draw [dashed] (v6) to (v8);
\draw [dashed] (v6) to (4,-1);
\draw [dashed] (v8) to (7.9,2.9);


\draw[step=1cm,gray,very thin] (-0.9,-0.9) grid (7.9,2.9);
\end{tikzpicture}
\caption{A counter example to monotonicity for every oder.}
\label{fig:notmono}
\end{figure} 
\end{remark}

With Theorem \ref{theorem 2D random walk reflection}, we could consider the problem of covering a nearest neighbor path in $\ZZ^d$. For any integer $N\ge 1$, let $\partial B_1(0,N)$ be the boundary of the $L_1$ ball in $\ZZ^d$ with radius $N$. We say that a nearest neighbor path
$$
\mathcal{P}=\big(P_0,P_1,\cdots, P_K\big)
$$
connects 0 and $\partial B_1(0,N)$ if $P_0=0$ and $\inf\{n: \|P_n\|_1= N\}=K$. And we say that a path $\mathcal{P}$ is covered by the first $L$ steps of $\{X_n\}_{n=0}^\infty$ if
$$
{\rm Trace}(\mathcal{P})\subseteq {\rm Trace}(X_0,X_1,\cdots, X_L).
$$
Then we are able to use Theorem \ref{theorem 2D random walk reflection} to show that the covering probability of any such path can be bounded by that of the diagonal. 
\begin{theorem}
\label{theorem D random walk}
For each integers $L\ge N\ge 1$, let $\mathcal{P}$ be any nearest neighbor path in $\ZZ^d$ connecting 0 and $\partial B_1(0,N)$. $X_n, n\ge 0$ be a $d$ dimensional simple random walk starting at 0. Then
$$
P\big({\rm Trace}(\mathcal{P})\in {\rm Trace}(X_0,\cdots, X_L)\big)\le P\big(\overset{\nearrow}{\mathcal{P}}\in {\rm Trace}(X_0,\cdots, X_L)\big)
$$
where 
$$
\overset{\nearrow}{\mathcal{P}}=\Big({\rm arc}_1[0:d-1],{\rm arc}_2[0:d-1],\cdots, {\rm arc}_{[N/d]}[0:d-1], {\rm arc}_{[N/d]+1}[0:N-d[N/d]] \Big)
$$
where
$$
{\rm arc}_1[0:d-1]=\left(0, e_1, e_1+e_2,\cdots, \sum_{i=1}^{d-1}e_i\right) 
$$
and ${\rm arc}_k=(k-1)\sum_{i=1}^{d}e_i+{\rm arc}_1$. 
\end{theorem} 

The following main theorem gives an upper bound of the covering probability over all nearest neighbor paths connecting 0 and $\partial B_1(0,N)$.

\begin{theorem}
\label{Theorem Main}
There is $P_d\in (0,1)$ such that for any nearest neighbor path $\mathcal{P}=(P_0,P_1,\cdots, P_K)$ connecting 0 and $\partial B_1(0,N)$ and $\{X_n\}_{n=0}^\infty$ which is a $d$ dimensional simple random walk starting at 0 with $d\ge 4$, we always have 
$$
P\left({\rm Trace}(\mathcal{P})\subseteq {\rm Trace}\big(\{X_n\}_{n=0}^\infty\big) \right)\le P_d^{[N/d]}.  
$$
Here $P_d$ equals to the probability that $\{X_n\}_{n=0}^\infty$ ever returns to the $d$ dimensional diagonal line.  
\end{theorem}

Note that in Theorem  \ref{Theorem Main} the upper bound is not very sharp since we only look at returning to the exact diagonal line for $[N/d]$ times, which may cover at most $1/d$ of the total points in $\overset{\nearrow}{\mathcal{P}}$. Although for any fixed $d$, we still have an exponential decay with respect to $N$, when $d\to\infty$, such exponential decaying speed, which is lower bounded by $\left(\frac{1}{2d}\right)^{1/d}$, goes to one. Fortunately, in Appendix \ref{sec:appb} we are able to show that $\lim_{d\rightarrow\infty} 2dP_d=1$, and then further find an upper bound on the asymptotic of the probability that a $d$ dimensional simple random walk starting from some point in Trace$\big(\overset{\nearrow}{\mathcal{P}}\big)$ will ever return to Trace$\big(\overset{\nearrow}{\mathcal{P}}\big)$. Note that we now need to return at least $N$ time to cover all the points in $\overset{\nearrow}{\mathcal{P}}$. We state this result as an additional theorem which is stronger than Theorem \ref{Theorem Main} since the proof is much more elaborate and for many uses Theorem \ref{Theorem Main} is sufficient.

\begin{theorem}
\label{Theorem sharp}
There is a $C\in (0,\infty)$ such that for any $d\ge 4$ and any nearest neighbor path $\mathcal{P}=(P_0,P_1,\cdots, P_K)\subset \ZZ^d$ connecting 0 and $\partial B_1(0,N)$ and $\{X_n\}_{n=0}^\infty$ which is a $d$ dimensional simple random walk starting at 0, we always have 
$$
P\left({\rm Trace}(\mathcal{P})\subseteq {\rm Trace}\big(\{X_n\}_{n=0}^\infty\big) \right)\le \left(\frac{C}{d}\right)^N.  
$$
\end{theorem}
The proof of Theorem \ref{Theorem sharp} can be found at the end of Section \ref{appendix A 1}.

\begin{remark}
Actually, any $C>3/2$ will serve as a good upper bound for sufficiently large $d$. See Remark \ref{remark  tedious} in Appendix \ref{sec:appb} for details. 
\end{remark}

\begin{remark}
Note that we do not present a proof of Theorem \ref{Theorem sharp} for $d=3$. With Theorem \ref{theorem D random walk} at hand, it is possible to prove this case since the cover probability of the diagonal path is bounded by the probability a random walk returns n times to the diagonal before leaving the ball $B_1(n)$. By using concentration inequalities after truncation of the excursion times of 2 dimensional random walks (not simple). However this yields an extremely non sharp bound with a slower order of decaying rate. 
\end{remark}

%


We conjecture that the upper bound in Theorem \ref{Theorem Main} can be improved for paths of length $N^{1+\epsilon}$, for small enough $\epsilon>0$. If this conjecture holds this will yield non trivial bounds for the chemical distance within a 3 dimension random walk trace:
\begin{conjecture}
\label{sharp conjecture}
There is an $\epsilon>0$ and a $\hat P_d\in (0,1)$ such that for any nearest neighbor path $\mathcal{P}$, connecting $0$ and $\partial B_1(N)$, such that $|\CP|\le N^{1+\epsilon}$ and $\{X_n\}_{n=0}^\infty$ which is a $d$ dimensional simple random walk starting at 0 with $d\ge 4$, we always have 
$$
P\left({\rm Trace}(\mathcal{P})\subseteq {\rm Trace}\big(\{X_n\}_{n=0}^\infty\big) \right)\le   \hat P_d^{N^{1+\epsilon}}.  
$$
\end{conjecture}

\begin{remark}\label{rem:sznitman}
Note that the probability to cover a space filling curve in $B_1(0,N)$ decays asymptotically slower that $c^{N^{d}}$.  Sznitman \cite[Section 2]{sznitmanvacant} showed that the probability a random walk path covers $B_1(0,N)$ completely can be bounded below by $ce^{-cN^{d-1}\log N}$. 
\end{remark}

A natural approach towards Conjecture \ref{sharp conjecture} is to try applying the same reflection process in this paper but also consider the repetition of visits rather than just looking at the trace of the path. However, it is shown in Section 6 that once we consider repetition, the diagonal line (with repetition) no longer maximizes the covering probability.  

\

 And for the family of monotonic nearest neighbor paths starting at 0, we also conjecture that the cover probability is minimized when the path goes straightly along a coordinate axis.  I.e., 

\begin{conjecture}
\label{monotonic conjecture}
For each integers $L\ge N\ge 1$, let $\mathcal{P}$ be any nearest neighbor monotonic path in $\ZZ^d$ with length $N$. $X_n, n\ge 0$ be a d dimensional simple random walk starting at 0. Then
$$
P\big({\rm Trace}(\mathcal{P})\subseteq {\rm Trace}(X_0,\cdots, X_L)\big)\ge P\big(\overset{\rightarrow}{\mathcal{P}}\subseteq {\rm Trace}(X_0,\cdots, X_L)\big)
$$
where 
$$
\overset{\rightarrow}{\mathcal{P}}=\Big((0,0,\cdots, 0),(1,0,\cdots, 0),\cdots,(N-1,0,\cdots, 0) \Big).
$$
\end{conjecture}

\begin{remark}
Note that the constants we get in Theorem \ref{Theorem Main} are not sharp. In fact, the upper bound we obtain for the covering of the diagonal path is of order $(1/2d)^N$. If we use the same argument as in Theorem \ref{Theorem Main} for the straight line we will get a bound of $[1/2(d-1)]^N$, since a return to the straight line is equivalent to a $d-1$ dimensional random walk returning to the origin. Thus we get that the bound we obtain is larger for the path that we conjecture minimizes the cover probability.
\end{remark}

 The structure of this paper is as follows: in Section 2 we prove a combinatorial inequality, which can be found later equivalent to finding a one-to-one mapping between paths of random walks within each of some equivalence classes. In Section 3 we use this combinatorial inequality to prove Theorem \ref{theorem 2D random walk reflection}. With Theorem \ref{theorem 2D random walk reflection}, we construct a finite sequence of paths with non-decreasing covering probabilities in Section 4 to show that the covering probability is maximized by the path that goes along the diagonal, see Theorem \ref{theorem D random walk} and \ref{theorem 2D random walk}. The proof of Theorem \ref{Theorem Main} is completed is Section 5, while in Section 6 we discuss the two conjectures and show numerical simulations. In Appendix \ref{sec:appb} we prove that $\lim_{d\rightarrow\infty} 2dP_d=1$ and then show that the probability a simple random walk returns to $\overset{\nearrow}{\mathcal{P}}$ also has an upper bound of $O(d^{-1})$, which implies Theorem \ref{Theorem sharp}. In Appendix \ref{sec:appa} we prove that the monotonicity fails when considering covering probability with repetitions.

\section{Combinatorial Inequalities}

In this section, we discuss a combinatorial inequality problem, which can be found equivalent to finding a one-to-one mapping between paths of random walks. For $n\in \ZZ^+$ and $\Omega$ be a set of $n$ integer numbers, say $\Omega=\{1,2,\cdots, n\}$ and any $A\subset\Omega$, abbreviate
$-A=\{-x:x\in A\}$ and $A^c=\Omega\setminus A$.
For any $m\in \ZZ^+$, consider a {\bf collection of arcs} which is a ``vector" of subsets 
$$
\vec V=V_1\otimes V_2\otimes\cdots \otimes V_m, \ V_k\subseteq \Omega,
$$
where each $V_k$ is called an arc, and an $m$ dimensional vector $\vec D=(\delta_1,\cdots, \delta_m)\in \{-1,1\}^{m}$ which is called a {\bf configuration}. Then we can introduce the {\bf inner product} 
\beq
\label{inner product} 
\vec D \cdot \vec V=\bigcup_{k=1}^m \delta_k V_k\subseteq -\Omega \cup \Omega.
\eeq
Moreover, for any subset $A\subseteq \Omega$, we say a configuration $\vec D$ of $\vec V$ {\bf covers} the reflection $A$ if 
$$
 -A^c\cup A\subseteq \vec D \cdot \vec V,
$$ 
and let 
$$
\mathcal{C}\left( \vec V,A\right)=\left\{\vec D: \  -A^c\cup A\subseteq \vec D \cdot \vec V \right\}
$$
be the subset of all such configurations. Then to show that a simple random walk has higher probability to cover the reflection of a finite set which bring it closer to the origin, we first need to prove that 
\begin{lemma}
\label{Lemma Combinatorial Inequalities}
For any $m,n\in \ZZ^+$, and any collection of arcs $\vec V$
\beq
\label{Combinatorial Inequalities}
\left| \mathcal{C}\left( \vec V,\Omega\right)\right|\ge \left| \mathcal{C}\left( \vec V, A\right)\right|
\eeq
for all $A\subseteq \Omega$.
\end{lemma}
\begin{proof}
First by symmetry one can immediately see that 
$$
\left| \mathcal{C}\left( \vec V,\Omega\right)\right|=\left| \mathcal{C}\left( \vec V,\O\right)\right|.
$$
So we will concentrate when $A\not=\O$ and prove the inequality by induction on $m$ and $n$. To show the basis of induction, it is easy to see that for any $n$ and $m=1$
$$
\left| \mathcal{C}\left( \vec V,\Omega\right)\right|=0=\left| \mathcal{C}\left( \vec V,A\right)\right|
$$
if $V_1\not=\Omega$ and 
$$
\left| \mathcal{C}\left( \vec V,\Omega\right)\right|=1>0=\left| \mathcal{C}\left( \vec V,A\right)\right|
$$
if $V_1=\Omega$. Then for any $m$ and $n=1$, by definition we must have $V_i=\{1\}$ or $\O$ for each $i$, and we always have $A=\Omega$. Thus 
$$
\left| \mathcal{C}\left( \vec V,\Omega\right)\right|=\left| \mathcal{C}\left( \vec V,A\right)\right|=2^{n_e(\vec V)}\left(2^{m-n_e(\vec V)}-1\right),
$$
where $n_e(\vec V)$ is the number of empty sets in $V_1,\cdots, V_m$. With the method of induction, suppose the desired inequality is true for all $n< n_0$ and all $n=n_0, m\le m_0$. Then for $n=n_0, m=m_0+1$, we can separate the last arc $V_{m_0+1}$ and the rest of the arcs and look at the truncated system at $m_0$. I.e., 
$$
\vec V[1:m_0]=V_1\otimes V_2\otimes\cdots \otimes V_{m_0}
$$
and 
$$
\vec D[1:m_0]=(\delta_1,\cdots, \delta_{m_0}).
$$
We have for any $A$
\beq
\label{separate}
\begin{aligned}
\mathcal{C}\left( \vec V,A\right)&=\left\{\vec D: -A^c\cup A\subseteq \vec D[1:m_0] \cdot \vec V[1:m_0] \right\}\cup\mathcal{P}\left( \vec V,A\right) 
\end{aligned}
\eeq
where 
$$
\mathcal{P}\left( \vec V,A\right) =\left\{\vec D: \  -A^c\cup A\nsubseteq \vec D[1:m_0] \cdot \vec V[1:m_0], \  -A^c\cup A\subseteq \vec D \cdot \vec V  \right\}.
$$
Noting that 
$$
\left| \left\{\vec D: \  -A^c\cup A\subseteq \vec D[1:m_0] \cdot \vec V[1:m_0] \right\}\right|=2\left| \mathcal{C}\left( \vec V[1:m_0] ,A\right)\right|, 
$$
and that the two events in \eqref{separate} are disjoint, 
\beq
\label{separate_1}
\left| \mathcal{C}\left( \vec V,A\right)\right|=2\left| \mathcal{C}\left( \vec V[1:m_0] ,A\right)\right|+\left|\mathcal{P}\left( \vec V,A\right) \right|.
\eeq
In order to study the cardinality of $\mathcal{P}$, we first show that 
\begin{lemma}
\label{lemma one-to-one 1}
For any two different $\vec D_1$ and $\vec D_2$ in $\mathcal{P}\left( \vec V,A\right)$, we must have 
$$
\vec D_1[1:m_0]\not=\vec D_2[1:m_0].   
$$
\end{lemma}
\begin{proof}
The proof is straightforward. Suppose $\vec D_1[1:m_0]=\vec D_2[1:m_0]=D'$, then their $m_0+1$st coordinates must be different. Thus
$$
V_{m_0+1}\cup D_1[1:m_0] \cdot \vec V[1:m_0]\supseteq -A^c\cup A
$$
and
$$
-V_{m_0+1}\cup D_1[1:m_0] \cdot \vec V[1:m_0]\supseteq -A^c\cup A.
$$
The first equality above implies that 
$$
 \Big(\vec D_1[1:m_0] \cdot \vec V[1:m_0]\Big)^c\cap \Big(-A^c\cup A\Big)\subseteq V_{m_0+1}
$$
while the second implies that 
$$
 \Big(\vec D_1[1:m_0] \cdot \vec V[1:m_0]\Big)^c\cap \Big(-A^c\cup A\Big)\subseteq -V_{m_0+1}.
$$
Combining the two inequalities gives us 
$$
 -A^c\cup A\subseteq \vec D_1[1:m_0] \cdot \vec V[1:m_0] 
$$
which contradicts with the definition of $\mathcal{P}\left( \vec V,A\right)$. Thus the proof is complete. 
\end{proof}
 With Lemma \ref{lemma one-to-one 1}, we now know that there is one-to-one mapping between each configuration in $\mathcal{P}\left( \vec V,A\right)$ and its first $m_0$ coordinates. And note that 
$$
\mathcal{P}\left( \vec V,A\right)=\left(\mathcal{P}\left( \vec V,A\right)\cap \{\delta_{m_0+1}=1\}\right)\cap \left(\mathcal{P}\left( \vec V,A\right)\cap \{\delta_{m_0+1}=-1\}\right).
$$
We have 
\beq
\label{separate_2}
\begin{aligned}
\left|\mathcal{P}\left( \vec V,A\right) \right|=\left|\mathcal{C}_1\left( \vec V[1:m_0] ,A\right)\right|+\left|\mathcal{C}_2\left( \vec V[1:m_0] ,A\right)\right|
\end{aligned}
\eeq
where 
$$
\mathcal{C}_1\left( \vec V[1:m_0] ,A\right)=\left\{ \vec D'\in \{-1,1\}^{m_0},  \O\not=\Big(\vec D'\cdot \vec V[1:m_0]\Big)^c\cap \Big(-A^c\cup A\Big)\subseteq V_{m_0+1} \right\}
$$
and
$$
\mathcal{C}_2\left( \vec V[1:m_0] ,A\right)=\left\{ \vec D'\in \{-1,1\}^{m_0},  \O\not=\Big(\vec D' \cdot \vec V[1:m_0]\Big)^c\cap \Big(-A^c\cup A\Big)\subseteq -V_{m_0+1} \right\}.
$$
Moreover, note that $\mathcal{C}\left( \vec V[1:m_0] ,A\right)$, $\mathcal{C}_1\left( \vec V[1:m_0] ,A\right)$ and $\mathcal{C}_2\left( \vec V[1:m_0] ,A\right)$ are disjoint and that 
\beq
\label{inequality_A}
\mathcal{C}\left( \vec V[1:m_0] ,A\right)\cup \mathcal{C}_1\left( \vec V[1:m_0] ,A\right)\cup \mathcal{C}_2\left( \vec V[1:m_0] ,A\right)\subseteq \mathcal{C}'\left( \vec V' ,A'\right)
\eeq
where 
$$
\vec V'=\left(V_1\cap V_{m_0+1}^c\right)\otimes \left(V_2\cap V_{m_0+1}^c\right)\otimes\cdots \otimes \left(V_{m_0}\cap V_{m_0+1}^c\right)
$$
with each $V'_k=V_k\cap V_{m_0+1}^c\subseteq \Omega'= V_{m_0+1}^c$, $A'=A\cap V_{m_0+1}^c$, and
$$
\mathcal{C}'\left( \vec V',A'\right)=\left\{\vec D\in  \{-1,1\}^{m_0}: \  -\left(\Omega'\cap A'^c\right)\cup A'\subseteq \vec D' \cdot \vec V' \right\}.
$$
In words, in order to be in one of the 3 disjoint subsets above, we must guarantee that all points in $\Omega'= V_{m_0+1}^c$ under reflection of $A$ are covered by the configuration  $\vec D'$ of $\vec V[1:m_0]$. To verify \eqref{inequality_A}, one can note that for any
$$
\vec D'\in \mathcal{C}\left( \vec V[1:m_0] ,A\right)\cup \mathcal{C}_1\left( \vec V[1:m_0] ,A\right)\cup \mathcal{C}_2\left( \vec V[1:m_0] ,A\right)
$$
we have
$$
 \vec D'\cdot \vec V[1:m_0]\supseteq \Big(-A^c\cup A\Big)\cap \Big(- V_{m_0+1}^c\cup V_{m_0+1}^c\Big)
$$
which implies 
\beq
\label{inequality_A1}
\begin{aligned}
 &\Big(\vec D'\cdot \vec V[1:m_0]\Big)\cap \Big(- V_{m_0+1}^c\cup V_{m_0+1}^c\Big)\\
 & \ \ \ \ \supseteq \Big(-A^c\cup A\Big)\cap \Big(- V_{m_0+1}^c\cup V_{m_0+1}^c\Big).
 \end{aligned}
\eeq
In \eqref{inequality_A1} we have the right hand side equals to 
$$
A'\cup -\Big(A^c\cap  V_{m_0+1}^c\Big)=-\left(\Omega'\cap A'^c\right)\cup A',
$$
and the left hand side equals to 
$$
\bigcup_{k=1}^{m_0}\left(\delta_k V_k\cap  \Big(- V_{m_0+1}^c\cup V_{m_0+1}^c\Big)\right).
$$
Noting that for each $k$
$$
\delta_k V_k\cap  \Big(-V_{m_0+1}^c\cup V_{m_0+1}^c\Big)=\delta_k \big(V_k\cap V_{m_0+1}^c \big),
$$
we have
$$
\bigcup_{k=1}^{m_0}\left(\delta_k V_k\cap  \Big(- V_{m_0+1}^c\cup V_{m_0+1}^c\Big)\right)=\vec D' \cdot \vec V' 
$$
which shows that $\vec D'$ also in $\mathcal{C}'\left( \vec V',A'\right)$.

 Specifically, when $A=\Omega$, note that for any $D'\in \mathcal{C}'\left( \vec V',\Omega'\right)$, $\Omega'=\vec D' \cdot \vec V'\subseteq D'\cdot \vec V[1:m_0]$. Thus
$$
\vec D'\cdot \vec V[1:m_0]\cup V_{m_0}=\Omega
$$ 
which implies that 
\beq
\label{equality_Omega}
\mathcal{C}\left( \vec V[1:m_0] ,\Omega\right)\cup \mathcal{C}_1\left( \vec V[1:m_0] ,\Omega\right)\cup \mathcal{C}_2\left( \vec V[1:m_0] ,\Omega\right)= \mathcal{C}'\left( \vec V' ,\Omega'\right).
\eeq
Combining \eqref{separate_1}-\eqref{equality_Omega} and the induction hypothesis, we have
$$ 
\begin{aligned}
\left| \mathcal{C}\left( \vec V,\Omega\right)\right|=&\left| \mathcal{C}\left( \vec V[1:m_0] ,\Omega\right)\right|+\left| \mathcal{C}\left( \vec V[1:m_0] ,\Omega\right)\right|+\left|\mathcal{C}_1\left( \vec V[1:m_0] ,\Omega\right)\right|+\left|\mathcal{C}_2\left( \vec V[1:m_0] ,\Omega\right)\right|\\
=&\left| \mathcal{C}\left( \vec V[1:m_0] ,\Omega\right)\right|+\left|\mathcal{C}'\left( \vec V' ,\Omega'\right) \right|\\
\ge&\left| \mathcal{C}\left( \vec V[1:m_0] ,A\right)\right|+\left|\mathcal{C}'\left( \vec V' ,A\right) \right|\\
\ge&\left| \mathcal{C}\left( \vec V[1:m_0] ,A\right)\right|+\left| \mathcal{C}\left( \vec V[1:m_0] ,A\right)\right|+\left|\mathcal{C}_1\left( \vec V[1:m_0] ,A\right)\right|+\left|\mathcal{C}_2\left( \vec V[1:m_0] ,A\right)\right|\\
=&\left| \mathcal{C}\left( \vec V,A\right)\right|.
\end{aligned}
$$
And thus the proof of Lemma \ref{Lemma Combinatorial Inequalities} is complete. 
\end{proof}

\section{Proof of Theorem \ref{theorem 2D random walk reflection}}

With the combinatorial inequality above, we can study the covering probability of simple random walks. Let $\mathcal{N}_L$ be the set of all nearest neighbor paths starting at 0 of length $L+1$ and consider 2 subsets of  $\mathcal{N}_L$ as follows:
$$
\mathcal{N}_{L,1}=\left\{\vec x\in \mathcal{N}_L, A_0\cup B_0\subseteq{\rm Trace}\big(\vec x\big)  \right\},
$$
and
$$
\mathcal{N}_{L,2}=\left\{\vec x\in \mathcal{N}_L, A_0\cup \varphi_l(B_0)\subseteq{\rm Trace}\big(\vec x\big)  \right\}.
$$

\begin{figure}[h!]
\centering
\begin{tikzpicture}
\tikzstyle{redcirc}=[circle,
draw=black,fill=myred,thin,inner sep=0pt,minimum size=2mm]
\tikzstyle{bluecirc}=[circle,
draw=black,fill=myblue,thin,inner sep=0pt,minimum size=2mm]

\node (v1) at (0,0) [bluecirc] {\small{$~$}};
\node (v2) at (1,0) [bluecirc] {\small{$~$}};
\node (v3) at (2,0) [bluecirc] {\small{$~$}};
\node (v4) at (3,0) [bluecirc] {\small{$~$}};
\node (v5) at (4,0) [bluecirc]{\small{$~$}};
\node (v6) at (5,0) [bluecirc] {\small{$~$}};
\node (v11) at (6,0) [bluecirc] {\small{$~$}};
\node (v7) at (5,1) [redcirc] {\small{$~$}};
\node (v8) at (6,1) [bluecirc] {\small{$~$}};
\node (v9) at (7,1) [bluecirc] {\small{$~$}};
\node (v10) at (6,2) [redcirc] {\small{$~$}};

\draw [thick] (v1) to (v2);
\draw [thick] (v2) to (v3);
\draw [thick] (v3) to (v4);
\draw [thick] (v4) to (v5);
\draw [thick] (v5) to (v6);
\draw [myred,thick] (v6) to (v7);
\draw [myred,thick] (v7) to (v8);
\draw [thick] (v8) to (v9);
\draw [myred,thick] (v8) to (v10);
\draw [thick] (v6) to (v11);
\draw [thick] (v11) to (v8);
\draw [dashed] (v6) to (v8);
\draw [dashed] (v6) to (4,-1);
\draw [dashed] (v8) to (7.9,2.9);

\draw [black, thick] plot [smooth, tension=1.2] coordinates { (0,0) (1,1) (2,-1) (1,0) (2,0) (3,0) (4,1.5) (5,1) (5.5,0.5)};
\draw [cyan, thick] plot [smooth, tension=1.2] coordinates { (5.5,0.5) (6,0.5) (7,0) (7.5,0.5) (7,1) (6,1)};
\draw [myred, thick] plot [smooth, tension=1.2] coordinates { (5.5,0.5) (5.5,1) (5,2) (5.5, 2.5) (6,2) (6,1)};

\draw [cyan, thick] plot [smooth, tension=1.2] coordinates {(6,1) (6.5,0.5) (6.5,-0.5) (5,-1) (5,0)};
\draw [myred, thick] plot [smooth, tension=1.2] coordinates {(6,1) (5.5,1.5) (4.5,1.5) (4,0) (5,0)};

\draw[step=0.5cm,gray,very thin] (-0.9,-0.9) grid (7.9,2.9);
\end{tikzpicture}
\caption{Combinatorial covering argument.}
\label{fig:notIR}
\end{figure}

Noting that the probabilities a simple random walk starting at 0 taking each $\vec x\in \mathcal{N}_L$ as its first $L$ steps are the same, it suffices to show that 
\beq
\label{number of paths}
|\mathcal{N}_{L,1}|\ge |\mathcal{N}_{L,2}|. 
\eeq
To prove \eqref{number of paths}, we need to first partition $\mathcal{N}_L$ into disjoint subsets, each of which serves as an equivalence class under the equivalence relation on $\mathcal{N}_L$ describe below.

 For each $\vec x=(x_0,x_1,\cdots, x_L)\in \mathcal{N}_L$, let $T_0=0$, $T_1=\inf\{n: x_n\in l\}$, and
$$
T_n=\inf\{n\ge T_{n-1}: x_n\in l\}
$$
for each integer $n\in [2,L]$ to be the time of the $n$th visit to $l$. Here we use the convention that $\inf\{\O\}=\infty$ and let $T_{L+1}=\infty$. Then for each $n=0,1,\cdots, L$, define the $n$th arc of $\vec x$ as 
$$
{\rm arc}\big(\vec x, n\big)=\vec x[T_{n}:T_{n+1}),
$$
where we use the convention that $\vec x[k:\infty)=\vec x[k:L]$, and $\vec x[0:0)=\vec x[\infty:\infty)=\O$. I.e., the $n$th arc is the piece of the path between the (possible) $n$th and $n+1$st visit to $l$. Since $\vec x$ is a nearest neighbor path, all points in each arc must be on the same side of $l$. Then for any $D=(\delta_1,\delta_2,\cdots,\delta_L)\in \{-1,1\}^{L}$ we can define a mapping $\varphi_{l,D}$ on $\mathcal{N}_L$ so that for any $\vec x$
$$
\varphi_{l,D}\big( \vec x\big)=\left( {\rm arc}\big(\vec x, 0\big), \varphi_l^{\frac{1-\delta_1}{2}}\left({\rm arc}\big(\vec x, 1\big) \right),  \varphi_l^{\frac{1-\delta_2}{2}}\left({\rm arc}\big(\vec x, 2\big) \right), \cdots,  \varphi_l^{\frac{1-\delta_L}{2}}\left({\rm arc}\big(\vec x, L\big) \right)\right).
$$
In words, we keep the part of the path the same until it first visits $l$. Then for the $n$th arc, we keep it unchanged if  $\delta_n=1$ and reflect it around $l$ if $\delta_n=-1$. By definition, it is easy to see that $\varphi_{l,D}\big( \vec x\big)\in \mathcal{N}_L$. And since 
$$
\varphi_{l,D}\circ \varphi_{l,D}\big( \vec x\big)\equiv \vec x, 
$$
$\varphi_{l,D}$ forms a bijection on $ \mathcal{N}_L$. Now we can introduce the equivalence relation on $\mathcal{N}_L$ previously mentioned. For each two $\vec x, \vec y\in \mathcal{N}_L$, we say $\vec x\sim \vec y$ if there exist a $D\in \{-1,1\}^{L}$ such that (see Figure \ref{fig:notIR})
$$
\varphi_{l,D}\big( \vec x\big)=\vec y.
$$
Then one can immediately check that if $\vec x\sim \vec y$ and $D$ is the configuration, then 
\beq
\label{equivalent1}
\vec x=\varphi_{l,D}\big( \vec y\big)\Rightarrow \vec y\sim \vec x
\eeq
and for $D_0\equiv 1$, 
\beq
\label{equivalent2}
\vec x=\varphi_{l,D_0}\big( \vec x\big)\Rightarrow \vec x\sim \vec x.
\eeq
Moreover, if 
$$
\varphi_{l,D_1}\big( \vec x\big)=\vec y, \ \varphi_{l,D_2}\big( \vec y\big)=\vec z, 
$$
let 
$$
D_3=\left(2*\ind_{\delta_{1,1}=\delta_{2,1}}-1,2*\ind_{\delta_{1,2}=\delta_{2,2}}-1,\cdots, 2*\ind_{\delta_{1,L}=\delta_{2,L}}-1\right)
$$
we have 
\beq
\label{equivalent3}
\varphi_{l,D_3}\big( \vec x\big)=\vec z\Rightarrow \vec x\sim \vec z.
\eeq
Combining \eqref{equivalent1}-\eqref{equivalent3}, we have that $\sim$ forms a equivalence relation on $\mathcal{N}_L$, where each path in $\mathcal{N}_L$ belongs to one equivalence class. Thus all the equivalence classes are disjoint from each other and there has to be a finite number of them, forming a partition of $\mathcal{N}_L$. We denote these equivalence classes as 
\beq
\label{equivalent class}
\mathcal{C}_{L,1},\mathcal{C}_{L,2},\cdots, \mathcal{C}_{L,J}
\eeq
where each of them can be represented by its specific element $\vec x_{k,+}, \ k=1,\cdots, J$, which is the unique path in each class that always stays on the left of $l$. 

 Then for each $k$, let $n_{k,1}<n_{k,2}<\cdots<n_{k,m_k}$ be all the $n$'s such that 
$$
\Big|{\rm Trace}\left({\rm arc}\big(\vec x_{k,+}, n\big)\right)\Big|>1. 
$$
Note that the only case when we have $\Big|{\rm Trace}\left({\rm arc}\big(\vec x_{k,+}, n\big)\right)\Big|=0$ is when $T_n=T_{n+1}=\infty$ and the only case when it equals to 1 is when $T_n=L$.
Then for any $\vec x\in \mathcal{C}_{L,k}$ and any $D_1, D_2$ such that 
$$
\varphi_{l,D_1}\big( \vec x_{k,+}\big)=\varphi_{l,D_2}\big( \vec x_{k,+}\big)=\vec x,
$$
we must have $\delta_{1,n_{k,i}}=\delta_{2,n_{k,i}}$ for all $i$'s. So we have a well defined onto mapping $f$ between $\mathcal{C}_{L,k}$ and $\{-1,1\}^{m_k}$ where each $\vec x$ such that 
$$
\vec x=\varphi_{l,D}\big( \vec x_{k,+}\big)
$$
for some $D$ is mapped to 
$$
f\big(\vec x\big)=(\delta_{2,n_{k,1}}, \delta_{2,n_{k,2}}, \cdots, \delta_{2,n_{k,m_k}}).
$$
Moreover, for any two configurations $D_1$ and $D_2$ such that $\delta_{1,n_{k,i}}=\delta_{2,n_{k,i}}$ for all $i$'s, we also must have 
$$
\varphi_{l,D_1}\big( \vec x_{k,+}\big)=\varphi_{l,D_2}\big( \vec x_{k,+}\big)
$$
The reason of that is for all $n$ not in $\{n_{k,i}\}_{i=1}^{m_k}$,
$$
\Big|{\rm Trace}\left({\rm arc}\big(\vec x_{k,+}, n\big)\right)\Big|\le1
$$
which means those arcs are either empty or just one point $x_{T_n}$ right on the diagonal, which does not change at all under any possible reflection. Thus we have proved that the mapping $f$
is a bijection between $\mathcal{C}_{L,k}$ and $\{-1,1\}^{m_k}$.

 At this point we have the tools we need and can go back to compare the two covering probabilities. Noting that the equivalence classes in \eqref{equivalent class} form a partition of $\mathcal{N}_L$. It suffices to show that for each $k\le J$
\beq
\label{number of paths2}
\left|\mathcal{N}_{L,1}\cap \mathcal{C}_{L,k}\right|\ge \left|\mathcal{N}_{L,2}\cap \mathcal{C}_{L,k}\right|. 
\eeq
For each class $\mathcal{C}_{L,k}$. First, if 
$$
(A_0\cup B_0)\cap l\nsubseteq {\rm Trace}(\vec x_{k,+}[T_1,\cdots, T_L])
$$
then one can immediately see 
$$
\left|\mathcal{N}_{L,1}\cap \mathcal{C}_{L,k}\right|= \left|\mathcal{N}_{L,2}\cap \mathcal{C}_{L,k}\right|=0. 
$$
Otherwise, let $\Omega_k=(A_0\cup B_0)\cap {\rm Trace}(\vec x_{k,+}[0:T_1])^c\cap l^c$ and $A_k=A_0\cap {\rm Trace}(\vec x_{k,+}[0:T_1])^c\cap l^c$. And let $n_k=|\Omega_k|$. We can also list all points in $\Omega_k$ as $\omega_1,\cdots, \omega_{n_k}$ and all points in $\varphi_l(\Omega_k)$ as $\omega_{-1},\cdots, \omega_{-n_k}$, where $\varphi_l(\omega_j)=\omega_{-j}$ for all $j$. Then it is easy to check that 
\beq
\label{number of paths3}
\mathcal{N}_{L,1}\cap \mathcal{C}_{L,k}=\left\{\vec x\in \mathcal{C}_{L,k},  \Omega_k\subseteq {\rm Trace}(\vec x)\right\}
\eeq
since all other points in $A_0\cup B_0$ are guaranteed to be visited by $\vec x_{k,+}[0:T_1]$ or $\vec x_{k,+}[T_1,\cdots, T_L]$ which are both shared over all paths in this equivalence class. And we also have 
\beq
\label{number of paths4}
\mathcal{N}_{L,2}\cap \mathcal{C}_{L,k}\subseteq \left\{\vec x\in \mathcal{C}_{L,k},   A_k\cup \varphi_l\big(\Omega_k\cap A_k^c\big)\subseteq {\rm Trace}(\vec x)\right\}
\eeq
since $A_k\subseteq A_0$, $\Omega_k\cap A_k^c\subseteq B_0$. Finally, define 
$$
V_{k,i}=\left\{ j: \omega_j\in \Omega_k\cap {\rm Trace}\left({\rm arc}\big(\vec x_{k,+}, n_{k,i}\big)\right)\right\},  \ i=1,2,\cdots, m_k
$$
and $\vec V_k=V_{k,1}\otimes V_{k,2}\otimes\cdots\otimes V_{k,m_k}$. Then by the constructions above we have for any $\omega_j\in \Omega_k$, $\omega_j\in {\rm Trace}(\vec x)$ if and only if there exists some $i$ such that $j\in V_{k,i}$ and $f(\vec x)[i]=1$. And similarly $\varphi_l(\omega_j)\in {\rm Trace}(\vec x)$ if and only if there exists some $i$ so that $j\in V_{k,i}$ and $f(\vec x)[i]=-1$. In combination, for any $\omega_j\in \Omega_k\cup \varphi_l(\Omega_k)$, $\omega_j\in {\rm Trace}(\vec x)$ if and only if there exists some $i$ such that $j\in f(\vec x)[i]V_{k,i}$. 
Then taking the intersections and letting $\bar \Omega_k=\{1,2,\cdots, n_k\}$ and 
$$
\bar A_k=\{j: \omega_j\in A_k\},
$$
we have
\beq
\label{number of paths5}
\left\{\vec x\in \mathcal{C}_{L,k},  \Omega_k\subseteq {\rm Trace}(\vec x)\right\}=\left\{\vec x\in \mathcal{C}_{L,k},  \ \bar\Omega_k\subseteq f(\vec x)\cdot\vec V_k\right\}
\eeq
and
\beq
\label{number of paths6}
\begin{aligned}
&\left\{\vec x\in \mathcal{C}_{L,k},   A_k\cup \varphi_l\big(\Omega_k\cap A_k^c\big)\subseteq {\rm Trace}(\vec x)\right\}\\
& \ \ \ \ =\left\{\vec x\in \mathcal{C}_{L,k}, \  -(\bar\Omega_k\cap \bar A_k^c)\cup \bar A_k\subseteq f(\vec x)\cdot\vec V_k\right\}.
\end{aligned}
\eeq
Noting that the mapping $f$ is a bijection between $\mathcal{C}_{L,k}$ and $\{-1,1\}^{m_k}$, 
\beq
\label{number of paths7}
\left|\left\{\vec x\in \mathcal{C}_{L,k},  \ \bar\Omega_k\subseteq f(\vec x)\cdot\vec V_k\right\} \right|=|\mathcal{C} (\vec V_k,\bar\Omega_k)|
\eeq
and 
\beq
\label{number of paths8}
\left|\left\{\vec x\in \mathcal{C}_{L,k}, \  -(\bar\Omega_k\cap \bar A_k^c)\cup \bar A_k\subseteq f(\vec x)\cdot\vec V_k\right\}\right|=|\mathcal{C}(\vec V_k,\bar A_k)|.
\eeq
Apply Lemma \ref{Lemma Combinatorial Inequalities} on $\bar\Omega_k$, $\vec V_k$ and $\bar A_k$. The proof of this theorem is complete. \qed

\begin{remark}
With exactly the same argument, we can also have the reflection theorem on reflecting over a line $y=x+n$, $n\ge 1$ and $x=-y+n$. 
\end{remark}

\section{Path Maximizing Covering Probability}
To apply Theorem \ref{theorem 2D random walk reflection} specifically on covering the trace of a nearest neighbor path in $\CP\subset\ZZ^2$ connecting 0 and $\partial B_1(0,N)$, we can assume without loss of generality that $P_K\in\partial B_1(0,N)$ in the first quadrant. For each such path with at least one point at $(x,y)$ such that $|x-y|\ge 2$, we can always reflect it around $l$ equals to $x=y+1$ or $y=x+1$ with the part of the path on the same side of $l$ as $0$ to be $A_0$ and the remaining points to be $\varphi_l(B_0)$ and always increase the probability of covering.

 Note that, after the reflection, $A_0\cup B_0$ is the trace of another nearest neighbor path, and we can reduce the total difference 
$$
\sum |x_i-y_i|
$$
by at least one in each step. After a finite number of steps, we will end up with a nearest neighbor path that stays within $\{|x-y|\le 1\}$. And among all those such paths, applying Theorem \ref{theorem 2D random walk reflection} for reflection over $x=y$, we can show that the covering probability is maximized when we move all the ``one step corners" to the same side of the diagonal, which itself gives us a monotonic path that stays within distance one above or below the diagonal. Thus we have the theorem as follows. 

\begin{theorem}
\label{theorem 2D random walk}
For each integers $L\ge N\ge 1$, let $\mathcal{P}$ be any nearest neighbor path in $\ZZ^2$connecting 0 and $\partial B_1(0,N)$. $X_n, n\ge 0$ be a 2 dimensional simple random walk starting at 0. Then
$$
P\big({\rm Trace}(\mathcal{P})\subseteq {\rm Trace}(X_0,\cdots, X_L)\big)\le P\big(\overset{\nearrow}{\mathcal{P}}\subseteq {\rm Trace}(X_0,\cdots, X_L)\big)
$$
where 
$$
\overset{\nearrow}{\mathcal{P}}=\Big((0,0),(0,1),(1,1),(1,2),\cdots,([N/2],N-[N/2]) \Big).
$$
\end{theorem}

\begin{proof}
As outlined above, we first show that 
\begin{lemma}
\label{lemma difference 1}
For each integers $L\ge N\ge 1$, let 
$$
\mathcal{P}=\Big((x_0,y_0),\cdots, (x_K,y_K) \Big)
$$
be any nearest neighbor path in $\ZZ^2$connecting 0 and $\partial B_1(0,N)$ with length $K+1\ge N+1$ where there is an $i\le n$ such that $|x_i-y_i|\ge 2$ and where $x_K\ge 0, y_K\ge 0$. $X_n, \ n\ge 0$ be a 2 dimensional simple random walk starting at 0. Then there exists a nearest neighbor  path $\mathcal{P}_1$ that keep staying within $\{|x-y|\le 1\}$ such that 
$$
P\big({\rm Trace}(\mathcal{P})\in {\rm Trace}(X_0,\cdots, X_L)\big)\le P\big({\rm Trace}(\mathcal{P}_1)\in {\rm Trace}(X_0,\cdots, X_L)\big).
$$
\end{lemma}
\begin{proof}
For any path $\mathcal{Q}$ with lenght $K+1$, define its {\bf total difference} as 
\beq
\label{total difference}
D_T(\mathcal{Q})=\sum_{(x_i,y_i)\in {\rm Trace}(\mathcal{Q})} |x_i-y_i|. 
\eeq
For each such path in this lemma, without loss of generality we can always assume there is some $i$ such that $x_i-y_i\ge 2$. Consider the line of reflection $l: x=y+1$ (otherwise consider $l: y=x+1$). It is easy to see that there is at least one point along this path on the right side of $l$. I.e.
$$
B_0'={\rm Trace}(\mathcal{P})\cap \{x>y+1\}\not=\O. 
$$
Define $A_0={\rm Trace}(\mathcal{P})\cap B_0^c$, $B_0=\varphi_{l}(B_0')\cap A_0^c$, and $\hat B_0'=\varphi_{l}(B_0)\subseteq B_0'$. Thus we have ${\rm Trace}(\mathcal{P})=A_0\cup B_0'$ and 
\beq
\label{difference 21}
P\big(A_0\cup B_0'\subseteq {\rm Trace}(X_0,\cdots, X_L)\big)\le P\big(A_0\cup \hat B_0'\subseteq {\rm Trace}(X_0,\cdots, X_L)\big).
\eeq
Applying Theorem  \ref{theorem 2D random walk reflection} on $A_0$ and $B_0$ gives us
\beq
\label{difference 22}
P\big(A_0\cup \hat B_0'\subseteq {\rm Trace}(X_0,\cdots, X_L)\big)\le P\big(A_0\cup B_0\subseteq {\rm Trace}(X_0,\cdots, X_L)\big).
\eeq
Then noting that $\mathcal{P}$ is a nearest neighbor path starting at 0 with length $K+1$, let $C_{K,i}$ be the equivalence class it belongs to under the relation $\sim$, and let $\mathcal{P}'=\vec x_i$ be the representing element of $C_{K,i}$ where all arcs are reflected to the left of $l$. Then it is easy to see that 
\beq
\label{difference 23}
{\rm Trace}(\mathcal{P}')=A_0\cup \varphi_{l}(B_0')= A_0\cup B_0.
\eeq
Combine \eqref{difference 21}-\eqref{difference 23}, 
\beq
\label{difference 2}
P\big(\mathcal{P}\subseteq {\rm Trace}(X_0,\cdots, X_L)\big)\le P\big(\mathcal{P}'\subseteq {\rm Trace}(X_0,\cdots, X_L)\big).
\eeq
Then note that for any $j$ such that $x_j-y_j\ge 2$,
$$
\varphi_{l}(x_j,y_j)=(y_i+1,x_j-1)\in \ZZ^2
$$
while 
$$
\|(x_j,y_j)\|_1\ge \|(y_i+1,x_j-1)\|_1
$$
for all $x_i\ge y_i+2$. Since $(x_K,y_K)$ is in the first quadrant, if in addition we also have $x_K\ge y_K+1$, then $\varphi_{l}(x_K,y_K)$ remains in the first quadrant with $\|\varphi_{l}(x_K,y_K)\|_1=N$. Thus the new nearest neighbor path $\CP'$ is also one connecting 0 and $\partial B_1(0,N)$. Moreover, 
$$
D_T(\mathcal{P})=\sum_{j:(x_j,y_j)\in B_0'}(x_j-y_j)+ \sum_{j:(x_j,y_j)\in A_0}|x_j-y_j|,
$$
while
$$
D_T(\mathcal{P}')=\sum_{j:(x_j,y_j)\in \hat B_0'}(x_j-y_j-2)+ \sum_{j:(x_j,y_j)\in A_0}|x_j-y_j|,
$$
which shows that $D_T(\mathcal{P}')\le D_T(\mathcal{P})-2$. Then if there is a point $(x_j',y_j')$ in the new path $\mathcal{P}'$ with $|x_j'-y_j'|\ge 2$, we can repeat the previous process and the covering probability is non-decreasing. Noting that for each time we decrease the total difference by at least 2 while $D_T(\mathcal{P})$ is a finite number, after repeating a finite number of times, we must end up with a nearest neighbor path where no point satisfies $|x-y|\ge 2$. Thus, we find a nearest neighbor path $\mathcal{P}_1$ connecting 0 and $\partial B_1(0,N)$ that keep staying within $\{|x-y|\le 1\}$  with a higher covering probability. 
\end{proof}
 With Lemma \ref{lemma difference 1}, note that for any nearest neighbor path connecting 0 and $\partial B_1(0,N)$ that keeps staying within $\{|x-y|\le 1\}$, we can always look at the part of it after its last visit to 0 and it has a higher covering probability. And note that such part has to contain a monotonic path from 0 to $\partial B_1(0,N)$. Thus it is sufficient to show that $\overset{\nearrow}{\mathcal{P}}$ has the highest covering probability over all nearest neighbor monotonic paths $\mathcal{P}_1$ connecting 0 and $\partial B_1(0,N)$ that keeps staying within $\{|x-y|\le 1\}$. We can show this by specifying what the trace of each such path looks like. Letting $N_0=N-[N/2]$ and $\bar\Omega=\{1,2,\cdots, N_0\}$. We have
\begin{lemma} 
For each nearest neighbor path $\mathcal{P}_1$ connecting 0 and $\partial B_1(0,N)$that keeps staying within $\{|x-y|\le 1\}$, we must have that 
$$
(j,j)\in {\rm Trace}(\mathcal{P}_1),  \ \forall j=1,2,\cdots, [N/2]
$$
and that 
$$
(j-1,j) {\rm \ or \ } (j,j-1) \in  {\rm Trace}(\mathcal{P}_1),  \ \forall j=1,2,\cdots, N_0. 
$$
\end{lemma}
\begin{proof}
The proof follows easy from the fact that $\mathcal{P}_1$ is a nearest neighbor path. For any such path
$$
\mathcal{P}_1=\Big((x_0,y_0),\cdots, (x_K,y_K) \Big)
$$
Since $\mathcal{P}_1$ is a nearest neighbor path between 0 and $\partial B_1(0,L)$ that stays within $\{|x-y|\le 1\}$, thus for any $n\le N$ there must be at least one point belongs to the trace of $\mathcal{P}_1$ with $L_1$ norm $n$. Then note that in  $\{|x-y|\le 1\}$ the only integer point with $L_1$ norm $2j$ is $(j,j)$, while the only two points with $L_1$ norm $2j-1$ are $(j-1,j)$ and $(j,j-1)$. Thus the proof is complete. 
\end{proof}
 With the lemma above, for each such $\mathcal{P}_1$ define 
$$
\bar A=\big\{j: (j-1,j)\in{\rm Trace}(\mathcal{P}_1)\big\}\subseteq \bar\Omega
$$
and 
$$
\bar B=\big\{j: (j,j-1)\in{\rm Trace}(\mathcal{P}_1)\big\}\subseteq \bar\Omega.
$$
Then we have $\bar A\cup \bar B=\bar \Omega$, so that
$$
\mathcal{P}_1\supseteq\{(j,j), \ j=0,1,\cdots [N/2]\}\cup \{(j-1,j), \ j\in \bar A \}\cup \{(j,j-1), \ j\in\bar\Omega\cap \bar A^c \}.
$$
Define 
$$
A_0=\{(j,j), \ j=0,1,\cdots [N/2]\}\cup \{(j-1,j), \ j\in \bar A \}, \ B_0=\varphi_{l_0}(\{(j,j-1), \ j\in\bar\Omega\cap \bar A^c \}),
$$
where $l_0$ is the line $x=y$. Then 
$$
P\big({\rm Trace}(\mathcal{P}_1)\subseteq {\rm Trace}(X_0,\cdots, X_L)\big)\le P\big(A_0\cup \varphi_{l_0}(B_0)\subseteq {\rm Trace}(X_0,\cdots, X_L)\big).
$$
And note that
$$
A_0\cup B_0=\Big((0,0),(0,1),(1,1),(1,2),\cdots,([N/2],N-[N/2]) \Big)=\overset{\nearrow}{\mathcal{P}},
$$
which itself gives a monotonic nearest neighbor path connecting 0 and $\partial B_1(0,N)$. So Theorem  \ref{theorem 2D random walk reflection} finishes the proof. 
\end{proof}

\

 For fixed $N$, the inequality in Theorem \ref{theorem 2D random walk} gets equal when $L=\infty$ since the 2 dimensional simple random walk is recurrent and both probabilities goes to one. However, we can easily generalize the same result to higher dimensions. This will similarly give us Theorem \ref{theorem D random walk}.

\

\begin{proof}[Proof of Theorem \ref{theorem D random walk}]
This theorem can be proved by reflecting only on two coordinates in $\ZZ^d$ while keeping all the others unchanged. For any $n\ge 0$, we look at, without loss of generality, the subspace $l: a_1=a_2+l_0,  \ l_0\ge 0$ when $d\ge 3$, and define reflection $\varphi_l$ over $l$ as follows: for each point $(a_1,\cdots, a_d)\in \ZZ^d$, 
$$
\varphi_l(a_1,\cdots, a_d)=(a_2+l_0,a_1-l_0, a_3,\cdots, a_d). 
$$
Then for all paths in $\mathcal{N}_L$ (all nearest neighbor paths starting at $0$ of length $L+1$), we can again define $T_0=0$, $T_1=\inf\{n: x_n\in l\}$, and
$$
T_n=\inf\{n\ge T_{n-1}: x_n\in l\}
$$
for each integer $n\in [2,L]$ to be the time of the $n$th visit to subpsace $l$, and divide $\mathcal{N}_L$ into a partition of equivalence classes under $\varphi_{l,D}$ for all $D\in \{-1,1\}^L$. Then for each pair of disjoint finite subsets $A_0, B_0\in \{ x\le y+l_0\}$, let
$$
\mathcal{N}_{L,1}=\left\{\vec x\in \mathcal{N}_L, A_0\cup B_0\subseteq{\rm Trace}\big(\vec x\big)  \right\},
$$
and
$$
\mathcal{N}_{L,2}=\left\{\vec x\in \mathcal{N}_L, A_0\cup \varphi_l(B_0)\subseteq{\rm Trace}\big(\vec x\big)  \right\}.
$$
For each equivalence class $ \mathcal{C}_{L,k}$ as above, the exact same argument as in the proof of Theorem \ref{theorem 2D random walk reflection} guarantees that
$$
\left|\mathcal{N}_{L,1}\cap \mathcal{C}_{L,k}\right|\ge \left|\mathcal{N}_{L,2}\cap \mathcal{C}_{L,k}\right|. 
$$
So again we have 
\beq
\label{reflection d}
\begin{aligned}
&P\left(A_0\cup B_0\subseteq{\rm Trace}\big(\{X_n\}_{n=0}^L\big) \right)\\
&  \ \ \ \ \ge P\left(A_0\cup \varphi_l(B_0)\subseteq{\rm Trace}\big(\{X_n\}_{n=0}^L\big) \right).
\end{aligned}
\eeq
Then apply \eqref{reflection d} to any nearest neighbor path connecting 0 and $\partial B_1(0,N)$
$$
\mathcal{P}=(P_0,P_1,P_2,\cdots,P_L)
$$
And without loss of generality we can also assume that $P_L\in (\ZZ^+\cup\{0\})^d$. Let the subspace of reflection be $l: a_1=a_2+1$,  
$$
A_0={\rm Trace}(\mathcal{P})\cap \{a_1\le a_2+1\}, \  B_0'={\rm Trace}(\mathcal{P})\cap \{a_1> a_2+1\}.
$$
and
$$
B_0=\varphi_l(B_0')\cap A_0^c, \ \ \hat B_0'=\varphi_l(B_0).
$$
Without loss of generality we can assume $B_0'$ is not empty, note that ${\rm Trace}(\mathcal{P})=A_0\cup B_0'$, and that similar to the proof of Theorem \ref{theorem 2D random walk}, we can again let $\mathcal{P}'$ be the representing element in the equivalence class under $\sim$ that contains $\mathcal{P}$, which is another nearest neighbor path connecting 0 and $\partial B_1(0,N)$ where all the arcs are reflected to the same side of $l$ as 0. Then ${\rm Trace}(\mathcal{P}')=A_0\cup B_0$, and ${\rm Trace}(\mathcal{P})=A_0\cup B_0'\supseteq A_0\cup \hat B_0'$. By \eqref{reflection d} we have 
\beq
\label{reflection d1}
P\big({\rm Trace}(\mathcal{P})\in {\rm Trace}(X_0,\cdots, X_L)\big)\le P\big(\mathcal{P}'\in {\rm Trace}(X_0,\cdots, X_L)\big). 
\eeq
Moreover, define 
$$
\begin{aligned}
D_T(\mathcal{P})=\sum_{P_n\in {\rm Trace}(\mathcal{P})} \sum_{i,j\le d} |p_{n,i}-p_{n,j}|
\end{aligned}
$$ 
be the total difference of $\mathcal{P}$. Then note that for each $n$
$$
\sum_{i,j\le d} |p_{n,i}-p_{n,j}|=|p_{n,1}-p_{n,2}|+f_n(p_{n,1})+f_n(p_{n,2})+\sum_{3\le i,j\le d} |p_{n,i}-p_{n,j}|
$$
where 
$$
f_n(p)=\sum_{i=3}^d|p-p_{n,i}|
$$
which is a convex function of $p$. Thus, we rewrite 
$$
D_T(\mathcal{P})=\sum_{P_n\in A_0} \sum_{i,j\le d} |p_{n,i}-p_{n,j}|+\sum_{P_n\in B_0'} \sum_{i,j\le d} |p_{n,i}-p_{n,j}|
$$
and
$$
D_T(\mathcal{P}')=\sum_{P'_n\in A_0} \sum_{i,j\le d} |p'_{n,i}-p'_{n,j}|+\sum_{P'_n\in B_0} \sum_{i,j\le d} |p'_{n,i}-p'_{n,j}|
$$
For each $n$ such that $P'_n=(p_{n,1},\cdots,p_{n,d})\in A_0$, we always have 
$$
\sum_{i,j\le d} |p_{n,i}-p_{n,j}|=\sum_{i,j\le d} |p'_{n,i}-p'_{n,j}|.
$$
Otherwise, we must have $P'_n\in B_0$ and there must always be a $P_n=\varphi_l(P_n')\in \hat B_0'\subseteq B_0'$, which implies that 
$$
\sum_{i,j\le d} |p'_{n,i}-p'_{n,j}|=|p'_{n,1}-p'_{n,2}|+ f_n(p'_{n,1})+f_n(p'_{n,2})+\sum_{3\le i,j\le d} |p_{n,i}-p_{n,j}|.
$$
And since that $P_n\in B_0'$, $p_{n,1}\ge p_{n,2}+2$, so that for $p'_{n,1}=p_{n,2}+1$ and $p'_{n,2}=p_{n,1}-1$, we must have 
\beq
\label{d3}
\max\{p_{n,1},p_{n,2}\}>\max\{p'_{n,1},p'_{n,2}\},  \ \min\{p_{n,1},p_{n,2}\}<\min\{p'_{n,1},p'_{n,2}\},
\eeq
which implies that $|p'_{n,1}-p'_{n,2}|<|p_{n,1}-p_{n,2}|$. Then combining \eqref{d3}, and that $p'_{n,1}+p'_{n,2}=p_{n,1}+p_{n,2}$ with the fact that $f_n(p)$ is convex, we also have 
$$
 f_n(p'_{n,1})+f_n(p'_{n,2})\le f_n(p_{n,1})+f_n(p_{n,2})
$$
which further implies that $D_T(\mathcal{P})\ge D_T(\mathcal{P}')+1$. Again, noting that $D_T(\mathcal{P})$ itself is finite, so after at most a finite number of iterations, we will end up with a path $\mathcal{P}_1$ connecting 0 and $\partial B_1(0,N)$ within the region 
$$
R=\left\{(a_1,a_2,\cdots,a_d)\in \ZZ^d, \ \max_{i,j\le d}|a_i-a_j|\le 1\right\}. 
$$
And it is easy to see that for each point $\vec a_0=(a_{1,0},a_{2,0},\cdots,a_{d,0})$ in region $R$ and each subspace $l: a_{i}=a_{j}$, $\vec a_0'=\varphi_l(\vec a_0)$ must satisfy 
\beq
\label{reflection 1 over 0}
a'_{k,0}=\left\{
\begin{aligned}
&a_{k,0},\ \ {\rm if \ } k\not= i,j\\
&a_{j,0},\ \ {\rm if \ } k= i\\
&a_{i,0},\ \ {\rm if \ } k=j.
\end{aligned}
\right.
\eeq

\

Thus, apply the same argument as in the proof of Theorem \ref{theorem 2D random walk}: reflection over $a_2=a_1$, which reflects points in $\{a_2=a_1+1\}$ to $\{a_1=a_2+1\}$, and then for  reflections over $a_3=a_1, \ \cdots$ and then $a_d=a_1$. We will have a sequence of paths $\mathcal{P}_{2,i}, \ i=2,\cdots d$ in $R$ with covering probabilities that never decrease. Moreover, by the definition of our reflections, for each $n\le K$ and $2\le j\le d$ we have that $\{p_{2,i,n,1}\}_{i=2}^d$ is nondecreasing while $\{p_{2,i,n,j}\}_{i=2}^d$ is nonincreasing, and that 
$$
p_{2,j,n,1}\ge p_{2,j,n,j}, \ \forall 2\le j\le d.
$$
Thus for $\mathcal{P}_{2}=\mathcal{P}_{2,d}$, we must have 
\beq
\label{maximum 1}
p_{2,n,1}\ge \max_{2\le j\le d}p_{2,n,j}
\eeq
for all $n\le K$. Then we reflect $\mathcal{P}_{2}$ over $a_3=a_2$, $a_4=a_2, \ \cdots$, and $a_d=a_2$ which also gives us a sequence of paths $\mathcal{P}_{3,i}, \ i=3,\cdots d$ in $R$ with covering probabilities that never decrease. Letting $\mathcal{P}_{3}=\mathcal{P}_{3,d}$, similarly we must have 
\beq
\label{maximum 2}
p_{3,n,2}\ge \max_{3\le j\le d}p_{3,n,j}.
\eeq
Moreover recalling the formulas of reflections within $R$ in \eqref{reflection 1 over 0}, all reflections over  $a_i= a_2$, $i\ge 3$ will not change $\max_{2\le j\le d}p_{2,n,j}$ for any $n$. Thus, we still have \eqref{maximum 1} holds for $\mathcal{P}_{3}$. Repeating this process and we will have a sequence $\mathcal{P}_{4}, \mathcal{P}_{5},\cdots, \mathcal{P}_{d}$  with covering probabilities that never decrease, where each of them stays within $R$. And finally for $\mathcal{P}_{d}$, we must have 
\beq
\label{maximum i}
p_{d,n,i}\ge \max_{i+1\le j\le d}p_{d,n,j},
\eeq
for all $i\le d-1, \ n\le N.$ Noting again that $\mathcal{P}_{d}$ is a nearest neighbor path, then
$$
{\rm Trace}(\mathcal{P}_{d})\supseteq \mathcal{P}_{0}
$$
and the proof of this Theorem is complete. 
\end{proof}

\section{Proof of Theorem \ref{Theorem Main}}
 With Theorem \ref{theorem D random walk}, the proof of Theorem \ref{Theorem Main} follows immediately from the fact that the simple random walk on $\RR^d, \ d\ge 4$ returns to the one dimensional line $x_1=x_2=\cdots=x_d$ with probability less than 1. Note that for any nearest neighbor path $\mathcal{P}=(P_0,P_1,\cdots, P_N)$ and $\{X_n\}_{n=0}^\infty$ connecting $0$ and $\partial B_1(0,N)$ which is a $d$ dimensional simple random walk, letting 
$$
\mathcal{Q}=\left\{0, \sum_{i=1}^{d}e_i, 2\sum_{i=1}^{d}e_i, \cdots, [N/d]\sum_{i=1}^{d}e_i \right\}
$$
be the points in $\overset{\nearrow}{\mathcal{P}}$ on the diagonal, we always have by Theorem \ref{theorem D random walk},
\begin{align*}
P\left({\rm Trace}(\mathcal{P})\subseteq {\rm Trace}\big(\{X_n\}_{n=0}^\infty\big)\right) &\le P\left(\overset{\nearrow}{\mathcal{P}}\subseteq {\rm Trace}\big(\{X_n\}_{n=0}^\infty\big)\right) \\
&\le P\left(\mathcal{Q}\subseteq {\rm Trace}\big(\{X_n\}_{n=0}^\infty\big)\right). 
\end{align*}
Moreover, let $\{\tau_n\}_{n=1}^\infty$ be sequence of stopping times of all visits to diagonal line $\ell: x_1=x_2=\cdots= x_d$. Then
\beq
\label{asymptotic 1}
P\left(\mathcal{Q}\subseteq {\rm Trace}\big(\{X_n\}_{n=0}^\infty\big)\right)\le P(\tau_{[N/d]}<\infty). 
\eeq
To bound the probability on the right hand side of \eqref{asymptotic 1}. Define a new Markov process $\{Y_n\}_{n=0}^\infty\in \ZZ^d$, where 
$$
Y_n=(x_{n,1}-x_{n,2}, x_{n,2}-x_{n,3},\cdots, x_{n,d-1}-x_{n,d}).
$$
Note that we can also write $\tau_n=\inf\{n>\tau_{n-1}, Y_n=0\}$ and that $Y_n$ itself is a $d-1$ dimensional random walk with generator 
$$
\begin{aligned}
\mathcal{L}f(y)=&\frac{1}{2d}\sum_{i=1}^{d-1}f(y+e_i-e_{i+1})+f(y-e_1+e_{i+1})\\
&+f(y+e_1)+f(y-e_1)++f(y+e_d)+f(y-e_d)-f(y).
\end{aligned}
$$
With $d-1\ge 3$, we have $P(\tau_n<\infty|\tau_{n-1}<\infty)=P_d=1-G_Y(0)^{-1}<1$. And thus the proof of Theorem \ref{Theorem Main} complete. \qed

%
%

\section{Discussions}
In this section we discuss the conjectures and show numerical simulations.
\subsection{Covering Probabilities with Repetitions} 
In the proof of Theorem \ref{theorem D random walk}, note that each time we apply Theorem \ref{theorem 2D random walk reflection} and get a new path $\mathcal{P}'$ with higher covering probability, we always have 
$$
{\rm Trace}(\mathcal{P})=A_0\cup B_0'
$$
and
$$
{\rm Trace}(\mathcal{P}')=A_0\cup B_0
$$
where $B_0=\varphi_l(B_0')\cap A_0^c\subseteq \varphi_l(B_0')$. This, together with the fact that $A_0$ is disjoint with both $B_0$ and $B_0'$, implies that 
$$
|{\rm Trace}(\mathcal{P})|=|A_0|+|B_0'|\ge |A_0|+|B_0|=|{\rm Trace}(\mathcal{P}')|.
$$
In words, although the length of the path remains the same after reflection, the size of its trace may decrease. In fact, for any simple path connecting 0 and $\partial B_1(0,N)$, at the end of our sequence of reflections, we will always end up with a (generally non-simple) path whose trace is of size $N+1$. 

%

\

 One natural approach towards a sharper upper bound is taking the repetitions of visits in a non-simple path into consideration. For any path 
$$
\mathcal{P}=\big(P_0, P_1,\cdots, P_K \big)
$$
starting at 0 which may not be simple, and any point $P\in {\rm Trace}(\mathcal{P})$, we can define the $k$th visit to $P$ as $T_{1,P}=\inf\{n: P_n=P\}$ and
$$
T_{k,P}=\inf\{n> T_{k-1}: P_n=P\}
$$
with convention $\inf\{\O\}=\infty$. Then we can define the repetition of $P\in {\rm Trace}(\mathcal{P})$ in the path $\mathcal{P}$ as 
\beq
\label{repetition}
n_{P,\mathcal{P}}=\sup\{k: T_{k,P}<\infty\}
\eeq
and denote the collection of all such repetitions as $N_\mathcal{P}=\{n_{P,\mathcal{P}}: P\in  {\rm Trace}(\mathcal{P})\}$. It is to easy to see that $n_{P,\mathcal{P}}\equiv 1$ for all $P\in {\rm Trace}(\mathcal{P})$ when $\mathcal{P}$ is a simple path, and that 
$$
\sum_{P\in {\rm Trace}(\mathcal{P})}n_{P,\mathcal{P}}=K+1. 
$$

\

 For $d$ dimensional simple random walk $\{X_n\}_{n=0}^\infty$ starting at 0 and any point $P\in \ZZ^d$, we can again define the stopping times $\tau_{0,P}=0$, $\tau_{1,P}=\inf\{n: X_n=P\}$ and
\beq
\label{n stopping time}
\tau_{n,P}=\inf\{n>\tau_{n-1,P}: X_n=P\}
\eeq
with convention $\inf\{\O\}=\infty$. Then we have 
\begin{definition}
\label{visiting with repetitions}
For each nearest neighbor path $\mathcal{P}$, and $d$ dimensional simple random walk $\{X_n\}_{n=0}^\infty$, we say that $\{X_n\}_{n=0}^L$ covers $\mathcal{P}$ {\bf up to its repetitions} if 
$$
\tau_{n_{P,\mathcal{P}},P}\le L, \forall P\in {\rm Trace}(\mathcal{P}). 
$$
And we denote such event by ${\rm Trace}(\mathcal{P})\otimes N_\mathcal{P}\subseteq \{X_n\}_{n=0}^L$. 
\end{definition}
 Our {\bf hope} was, for any nearest neighbor path $\mathcal{P}$ and subspace of reflection like $l: x_i=x_j+l_0$, the probability of a simple random walk $\{X_n\}_{n=0}^L$ starting at 0 covering $\mathcal{P}$ up to its repetitions {\bf may} be upper bounded by that of covering the path $\mathcal{P}'$ up to its repetitions, where $\mathcal{P}'$ is the representing element in the equivalence class in $\mathcal{N}_K$ containing $\mathcal{P}$ under the reflection $\varphi_l$. In words, $\mathcal{P}'$  is the path we get by making all the arcs in $\mathcal{P}$ reflected to the same side as 0. 
 
 \

 Note that although $\mathcal{P}'$ may not be simple and the size of its trace could decrease, this will at the same time increase the repetition on those points which are symmetric to the disappeared ones correspondingly. In fact, under Definition \ref{visiting with repetitions}, the total number of points our random walk needs to (re-)visit is always 
$$
\sum_{P\in {\rm Trace}(\mathcal{P}')}n_{P,\mathcal{P}'}=K+1. 
$$
So if our previous guess were true, then we will be able to follow the same process as in Section 3 and 4 and end with the same path along the diagonal, but this time with a higher probability of being covered  up to its repetitions. Noting that the path we have in Theorem \ref{theorem D random walk} will visit points exactly on the diagonal line $x_1=x_2=\cdots=x_d$ for $O(K)$ times, the same construction of a $d-1$ dimensional random walk as in Section 5 will give the sharp upper bound we need in Conjecture \ref{sharp conjecture}.

\

 Unfortunately, here we present the counterexample and numerical simulations showing that Theorem \ref{theorem 2D random walk} and \ref{theorem D random walk} no longer holds for of certain non-monotonic paths. The idea of constructing those examples can be seen in the following preliminary model: Let $l$ be the line of reflection and suppose we have one point $x$ on the same side of $l$ as 0. Then suppose there is a equivalence class $\mathcal{C}_{L,k}$ with its representative element $\vec x_k$ having $2n$ arcs each visiting $x$ once. Then we look at the covering probability of $\{x, \varphi_l(x)\}\otimes (n, n)$ and that of its reflection $\{x\}\otimes 2n$. For the first one, we only need to choose $n$ of $2n$ arcs in $\vec x_k$ and reflect them to the other side while keeping the rest unchanged. So we have $\big(^{2n}_n\big)$ configurations. However, for the second covering probability which one may hope to be higher, the only configuration that may give us the covering up to this repetition is  $\vec x_k$ itself. Thus, at least in this equivalence class, the number of configurations covering $\{x, \varphi_l(x)\}\otimes (n, n)$ is larger than that of configurations covering $\{x\}\otimes 2n$.


\

With this idea in mind we give the following counterexample on actual 3 dimensional paths which shows precisely and rigorously that the covering probability is not increased after reflection. 

\noindent {\bf Counterexample 1.} Consider the following points $o=(0,0,0), \ y=(1,0,0)$, $z=(0,1,0), \ w=(1,1,0)\in \ZZ^3$, and paths 
$$
\mathcal{P}=(o,y,w,z)
$$
and 
$$
\mathcal{P}'=(o,y,w,y)
$$
which is the representative element of the equivalence class containing $\mathcal{P}$, under reflection over $l: x_2=x_1$. Let $\{X_n\}_{n=0}^\infty$ be a simple random walk starting at 0. Moreover, we use the notation $A=\ZZ^3\setminus\{y,z,w\}$ and define stopping times $\tau_a=\inf\{n\ge 1: \ X_n=a\}$ for all $a\in \ZZ^3$, and $\tau_A=\inf\{n\ge 1: \ X_n\notin A\}$. Thus we have 
\begin{Proposition}
\label{Counterexample 1}
For the paths $\mathcal{P}$ and $\mathcal{P}'$ defined above, 
\beq
\label{probability 1}
\begin{aligned}
&P\left({\rm Trace}(\mathcal{P})\otimes N_\mathcal{P}\subseteq \{X_n\}_{n=0}^\infty\right)\\
&  \ \ =2P_o(\tau_y=\tau_A)[P_o(\tau_y<\tau_\omega)+P_o(\tau_\omega<\tau_y)]P_o(\tau_y<\infty)\\
& \ \ \ +2P_o(\tau_\omega=\tau_A)P_o(\tau_y<\tau_z)P_o(\tau_y<\infty)\approx 0.08
\end{aligned}
\eeq
which is larger than 
\beq
\label{probability 2}
\begin{aligned}
&P\left({\rm Trace}(\mathcal{P}')\otimes N_{\mathcal{P}'}\subseteq \{X_n\}_{n=0}^\infty\right)\\
&  \ \ =P_o(\tau_y<\tau_\omega)[P_o(\tau_0<\tau_y)+P_o(\tau_y<\tau_0)]P_o(\tau_y<\infty) \\
&  \ \ \ +P_o(\tau_w<\tau_y)P_o(\tau_y<\infty)P_o(\tau_0<\infty)\approx 0.065.
\end{aligned}
\eeq
\end{Proposition}

 The proof of Proposition \ref{Counterexample 1} is basically a standard application of Green's functions for finite subsets. So we leave the detailed calculations in Appendix \ref{sec:appa}. For anyone who believes in law of larger numbers, we recommend them to look at the following numerical simulation which shows the empirical probabilities (which almost exactly agree with Proposition \ref{Counterexample 1}) of covering both paths with half a million independent paths of 3-dimensional simple random walks run up to $L=40000$.  

\begin{figure}[h]
\center
\includegraphics[width=0.65\columnwidth]{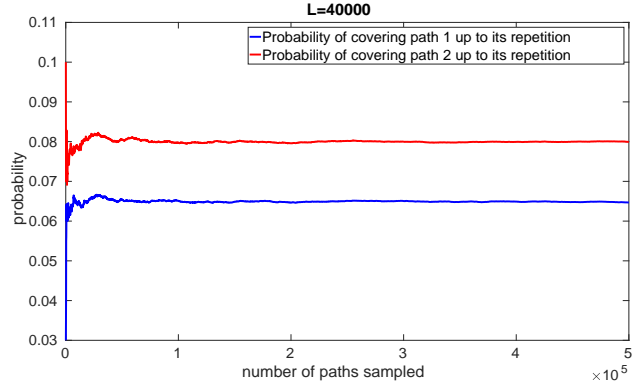} 
\caption{covering probabilities of path 1=$\mathcal{P}$ and path 2=$\mathcal{P}'$, $L=40000$}
\end{figure}

 For a finite length $\{X_n\}_{n=0}^L$ with $L<\infty$, although it is harder to calculate the exact covering probabilities theoretically, the following simulations on $L=4000, 400$ and $40$ show that the inequality in Proposition \ref{Counterexample 1} remains robust for fairly small $L$ (see Figure \ref{fig:counter2}). 

 At last, it is important to note that Proposition \ref{Counterexample 1} does not imply by any means a disproof of Conjecture \ref{sharp conjecture}. What it does show, however, is such a conjecture cannot be proved by the combinatorial argument on reflections shown in this paper.

\subsection{Monotonic Path Minimizing Covering Probability}

In Conjecture \ref{monotonic conjecture}, we conjecture that when concentrating on monotonic paths, the covering probability is minimized when the path takes a straight line along some axis. The intuition is, while all monotonic paths connecting 0 and $\partial B_1(0,N)$ have the same $L_1$ distance, the $L_2$ distance is maximized along the straight line, which makes it the most difficult to cover. This conjecture is supported for small $N$. In the following example, we have $d=3$ and $N=3$. By symmetry of simple random walk, one can easily see that for each monotonic path of length $N+1=4$, starting at 0, the covering probability must equal to that of one of the following five: 
$$
\begin{aligned}
&{\rm path 1: \ } (0,0,0)\to (1,0,0)\to(2,0,0)\to (3,0,0)\\
&{\rm path 2: \ } (0,0,0)\to (1,0,0)\to(2,0,0)\to (2,1,0)\\
&{\rm path 3: \ } (0,0,0)\to (1,0,0)\to(1,1,0)\to (1,2,0)\\
&{\rm path 4: \ } (0,0,0)\to (1,0,0)\to(1,1,0)\to (2,1,0)\\
&{\rm path 5: \ } (0,0,0)\to (1,0,0)\to(1,1,0)\to (1,1,1).\\
\end{aligned}
$$
The following simulation (see Figure \ref{fig:line}) shows that when $L=400$, the covering probability of path 1 is the smallest of them all. It should be easy to use the same calculation in Proposition \ref{Counterexample 1} to show the rigorous result when $L=\infty$.  
\begin{figure}[h!]
\center
\includegraphics[width=0.65\columnwidth]{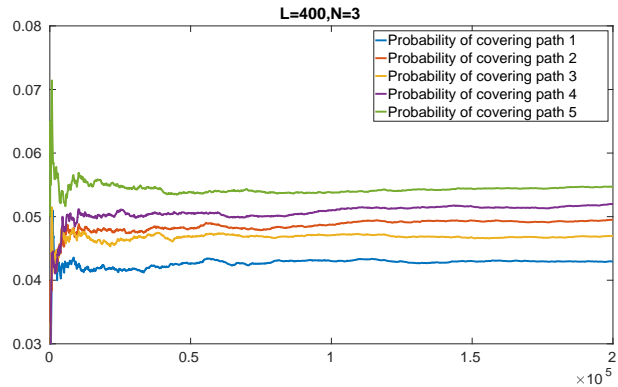} 
\caption{covering probabilities of monotonic paths starting at 0 of length 4\label{fig:line}}
\end{figure}

\begin{figure}[h!]
\center
\includegraphics[width=.65\columnwidth]{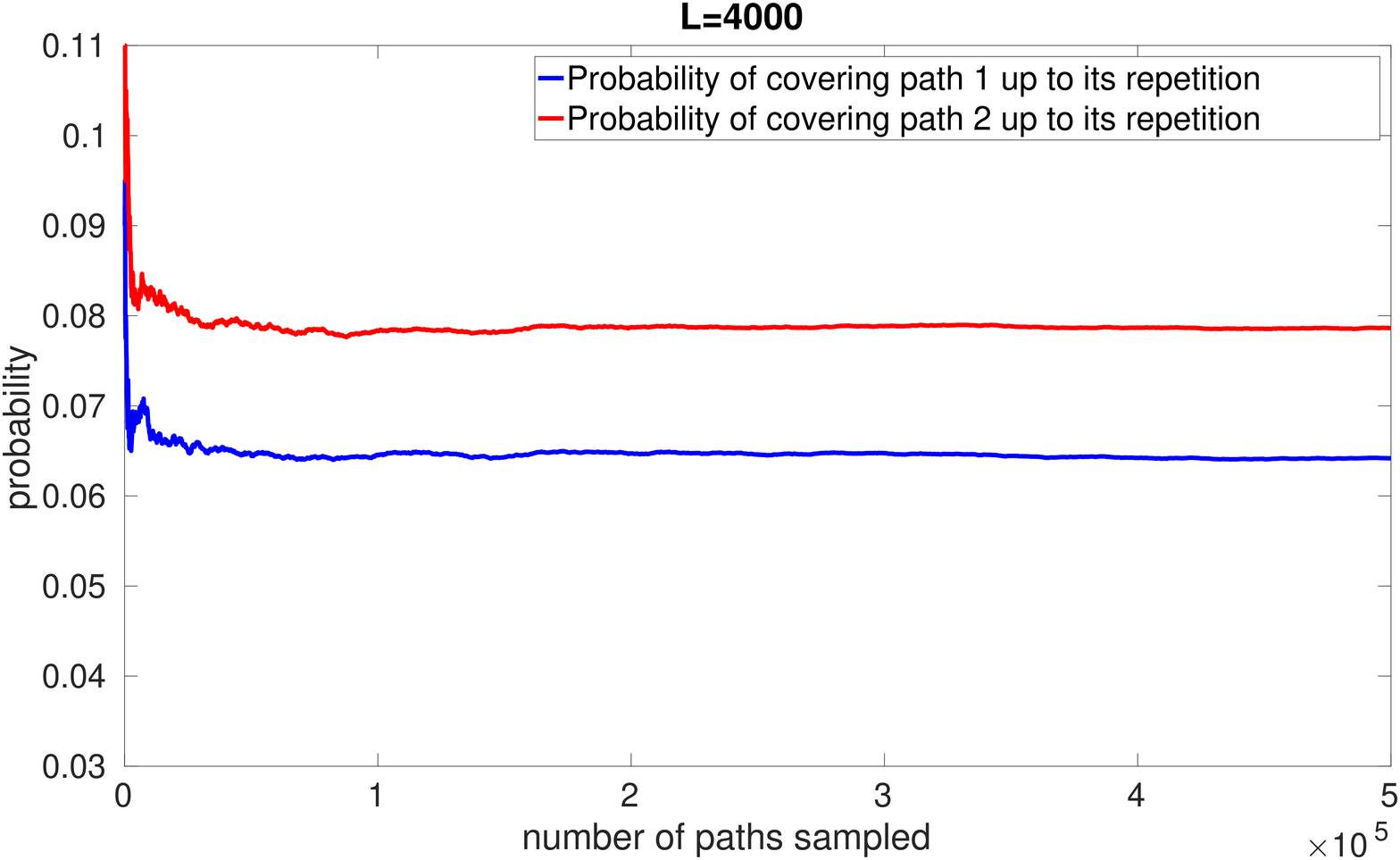} \\
\includegraphics[width=.65\columnwidth]{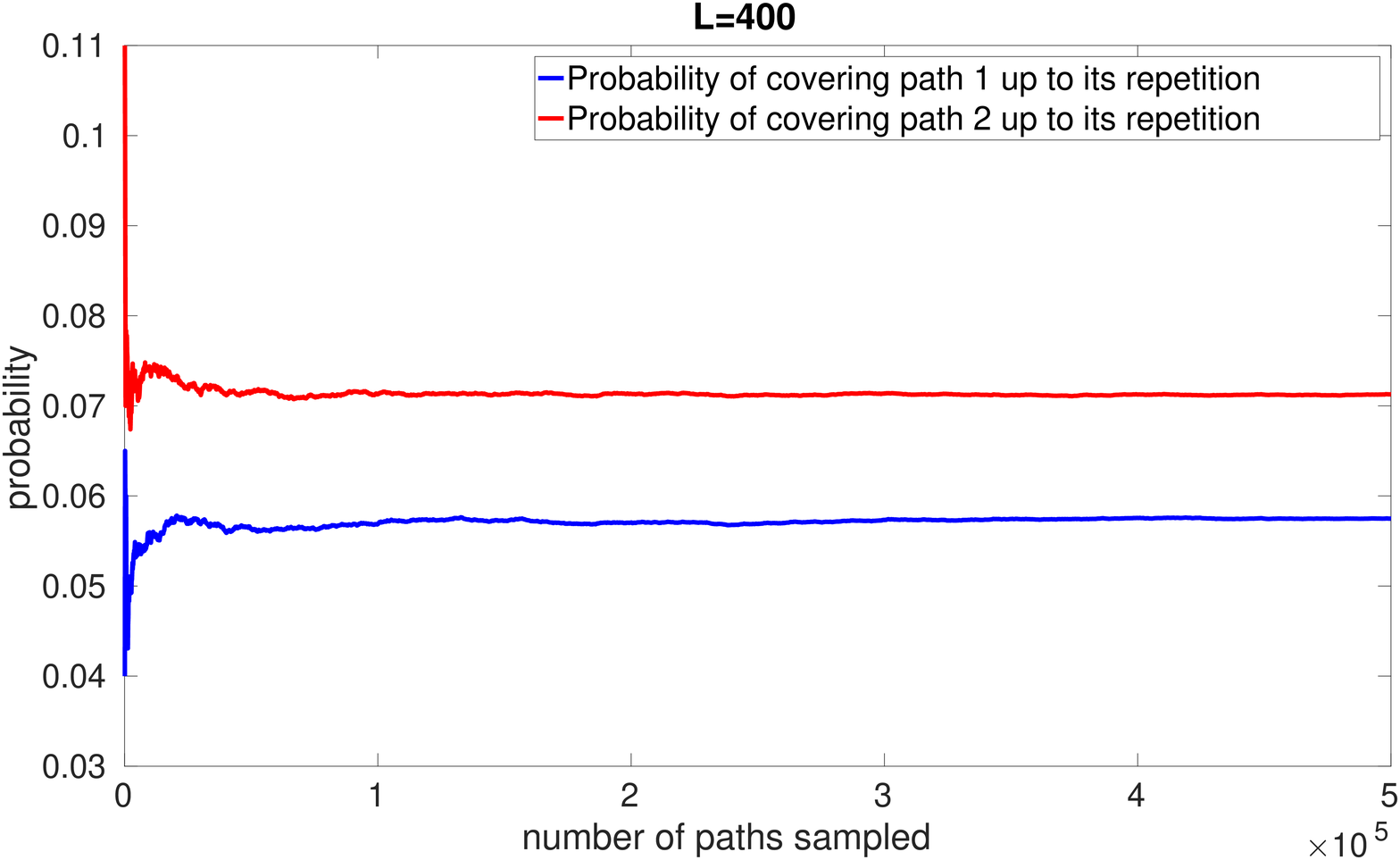} \\
\includegraphics[width=.65\columnwidth]{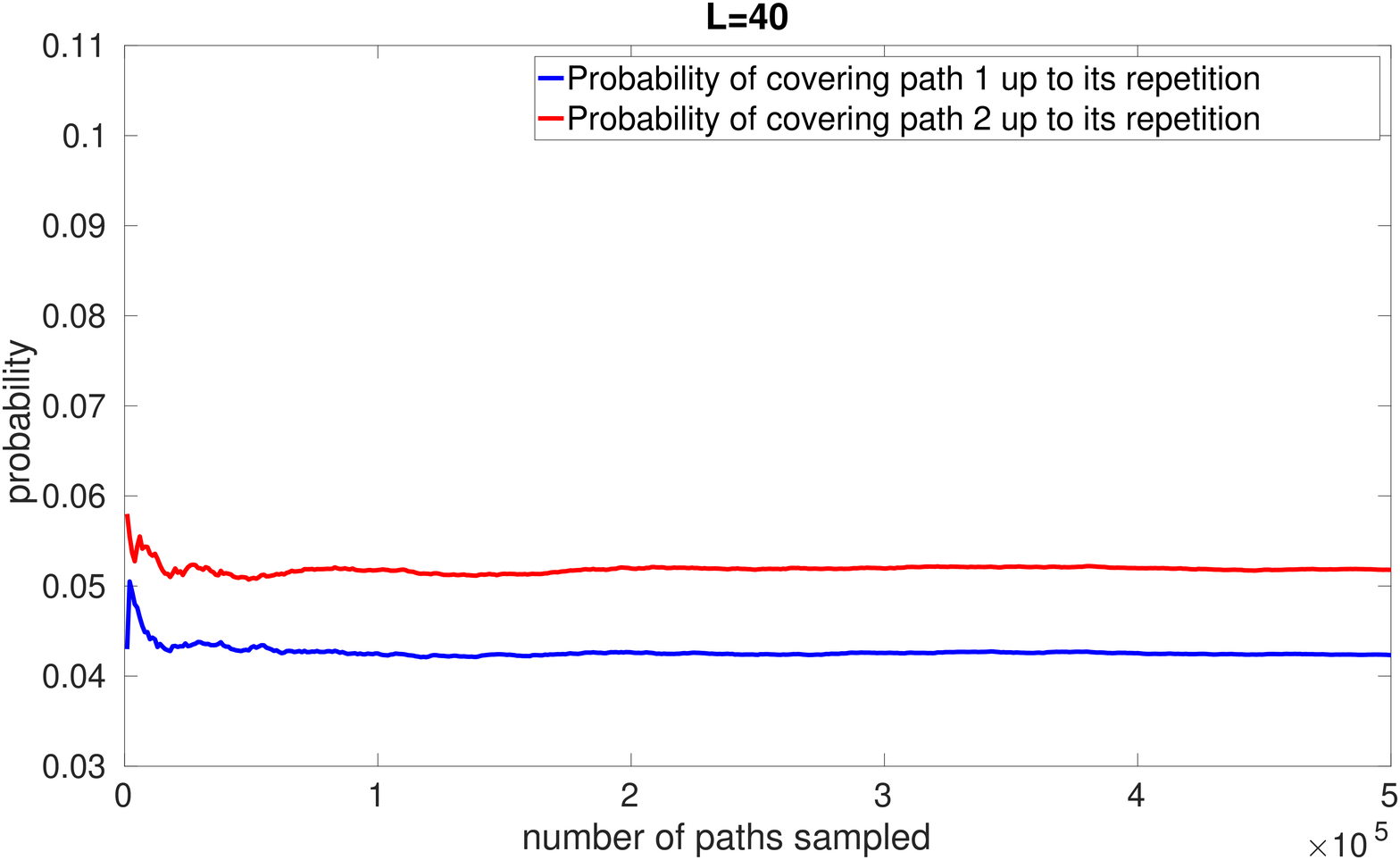} 
\caption{covering probabilities of path 1=$\mathcal{P}$ and path 2=$\mathcal{P}'$, $L=4000,400,40$\label{fig:counter2}}
\end{figure}


\bibliography{mono}
\bibliographystyle{plain}

%
%
%
%
%


\clearpage 
\appendix

\section{~}
\label{sec:appb}
\subsection{Introduction}
\label{appendix A 1} 
In this appendix, we find the asymptotic behavior of the returning probability (to the origin and diagonal line) of a $d$ dimensional simple random walk as $d\to \infty$. The asymptotics of the return to the origin is stated in \cite{Montroll}. We present an alternative proof here for two reasons, the first is that we believe the proof in \cite{Montroll} is not completely rigurous and the second is that the technique we show here can be generalized to the diagonal return probability.

For a $d$ dimensional simple random walk $\{X_{d,n}\}_{n=0}^\infty$ starting at 0 and any $x\in \ZZ^d$,  define the stopping time 
$$
\tau_{d,0}=\inf\{n\ge 1, \ X_{d,n}=x\}.
$$
Then the returning probability is defined by 
\beq
\label{returning probability}
p_d=P(\tau_{d,0}<\infty).
\eeq 
Our main result is as follows: 

\begin{theorem}
\label{Main Theorem}
For $p_d$ defined above, we have 
\beq
\label{main result}
\lim_{d\to\infty} 2dp_d=1.
\eeq
\end{theorem}

\noindent While Theorem \ref{Main Theorem} is stated for the leading order of $p_d$, one should be able to easily use the same technique in our proof to have the second and higher order asymptotics. 

\

\noindent Also, noting that 
$$
p_d=\frac{1}{2d}\sum_{i=1}^d (p_{d,e_{d,i}}+p_{d,-e_{d,i}})=p_{d,e_{d,1}},
$$
there is always a simple lower bound $p_d>1/(2d)$. With Theorem \ref{Main Theorem}, we actually know that such trivial lower bound is asymptotically sharp. Intuitively, when the dimension is very high, if a random walk does not return to 0 immediately in the first 2 steps, it will with high conditional probability get lost in the ocean of choices and never be able to go back to 0.  

\

\noindent Moreover, a similar method as in Theorem \ref{Main Theorem} may also work for other random walks. Particularly, for a specific $d-1$ dimensional one defined by
$$
\hat X_{d-1,n}=\left(X_{d,n}^1-X_{d,n}^2,X_{d,n}^2-X_{d,n}^3,\cdots, X_{d,n}^{d-1}-X_{d,n}^{d}\right)
$$
where $X_{d,n}^i$ is the $i$th coordinate of $X_{d,n}$, we can show the same asymptotic for $\hat X_{d-1,n}$ which also gives us the asymptotic of the probability that a $d$ dimensional simple random walk ever returns to the diagonal line.  To make the statement precise, consider the diagonal line in $\ZZ^d$
$$
l_d=\{(n,n,\cdots, n)\in \ZZ^d, n\in\ZZ\}\subset \ZZ^d.
$$
Define the stopping time
$$
\tau_{d,l_d}=\inf\{n\ge 1, \  X_{d,n}\in l_d\},
$$ 
and let
$$
P_{d}=P(\tau_{d,l_d}<\infty)
$$
be the returning probability to $l_d$. 

\begin{theorem}
\label{Theorem diagonal}
For $P_d$ defined above, we have 
\beq
\label{main result}
\lim_{d\to\infty} 2dP_{d}=1. 
\eeq
\end{theorem} 

With Theorem \ref{Theorem diagonal}, we further look at the probability that a $d$ dimensional simple random walk starting from some point in Trace$\big(\overset{\nearrow}{\mathcal{P}}\big)$ will ever return to Trace$\big(\overset{\nearrow}{\mathcal{P}}\big)$. Note that for each point 
$$
x=(x^{(1)},x^{(2)},\cdots, x^{(d)})\in {\rm Trace}\big(\overset{\nearrow}{\mathcal{P}}\big)
$$
we must have either $x=0$ or there must be some $1\le k\le d$ and $0\le n\le [N/d]$ such that 
$$
x^{(i)}=\left\{
\begin{aligned}
&n+1 \ \ i\le k\\
&n \hspace{0.39 in} i>k.  
\end{aligned}
\right.
$$
Thus when looking at $\hat x=(x^{(1)}-x^{(2)},x^{(2)}-x^{(3)}\cdots, x^{(d-1)}-x^{(d)})$ we must have either $\hat x=0$ or $\hat x=e_{d-1,i}$ for some $i=1,2,\cdots, d$. In this appendix, we will use the notation $e_{d-1,0}=0$ and let $D_{d-1}=\{e_{d-1,i}: i=0,1,\cdots, d-1\}\subset\ZZ^d$. One can immediately see that when simple random walk $X_{d,n}$ starting from some point in Trace$\big(\overset{\nearrow}{\mathcal{P}}\big)$ returns to Trace$\big(\overset{\nearrow}{\mathcal{P}}\big)$, we must have the corresponding non simple random walk $\hat X_{d-1,n}$ starting from $D_{d-1}$ returns to $D_{d-1}$. Thus for any simple random walk $X_{d,n}$ starting at 0, define the stopping times $T_{d,0}=0$
$$
T_{d,1}=\inf\left\{n\ge 1: \ X_{d,n}\in {\rm Trace}\big(\overset{\nearrow}{\mathcal{P}}\big)\right\},
$$
and 
$$
T_{d,k}=\inf\left\{n\ge T_{d,k-1}: \ X_{d,n}\in {\rm Trace}\big(\overset{\nearrow}{\mathcal{P}}\big)\right\}
$$
for all $k\ge 2$ with the convention $\inf\{n\ge \infty\}=\infty$. And for $\hat X_{d-1,n}$ also starting at 0, and any $0\le i,j\le d-1$, define the stopping time 
$$
T^{(i,j)}_{d-1}=\inf\{n\ge 1 X_{d-1,n}=e_{d-1,j}-e_{d-1,i}\}.
$$
Then it is easy to see that for any $k\ge 0$
\beq
\label{inequality diagonal 1}
P(T_{d,k+1}<\infty|T_{d,k}<\infty)\le \sup_{0\le i\le d-1} P\left(\inf_{0\le j\le d-1}\left\{T^{(i,j)}_{d-1}\right\}<\infty\right).
\eeq
With basically similar but more complicated technique as in the proof of Theorem \ref{Theorem diagonal} we have 
\begin{theorem}
\label{Theorem return d}
There is a $C<\infty$ such that for all $d\ge 4$, 
\beq
\label{inequality  sharp}
\sup_{0\le i\le d-1} P\left(\inf_{0\le j\le d-1}\left\{T^{(i,j)}_{d-1}\right\}<\infty\right)\le \frac{C}{d}. 
\eeq
\end{theorem}

With Theorem \ref{Theorem return d}, the proof of Theorem \ref{Theorem sharp} is imminent. 
\begin{proof}[Proof of Theorem \ref{Theorem sharp}] 
\label{Proof Theorem sharp}
With \eqref{inequality diagonal 1} and \eqref{inequality  sharp}, we can immediately have
$$
P\big(\overset{\nearrow}{\mathcal{P}}\in {\rm Trace}(X_0,X_1\cdots)\big)\le  P(T_{d,N}<\infty)
$$
while 
$$
P_x(T_{d,N}<\infty)\le \left[\sup_{0\le i\le d-1} P\left(\inf_{0\le j\le d}\left\{T^{(i,j)}_{d-1}\right\}<\infty\right)\right]^N\le\left(\frac{C}{d}\right)^N. 
$$
And the proof of Theorem \ref{Theorem sharp} is complete. 
\end{proof}

\subsection{Useful Facts from Calculus}

In this section, we list some very basic but useful facts from calculus that we are going to use later in the proof. 

\

\noindent  \textcircled{1} For any function $f(x)\in C(\RR)$ and any $a\in \RR$, 
\beq
\label{shift cos(x)}
\int_a^{a+2\pi}f\big(\cos(x),\sin(x)\big)dx=\int_0^{2\pi}f\big(\cos(x),\sin(x)\big)dx.
\eeq

\noindent  \textcircled{2} for any nonnegative integers $m,n$ and 
$$
C_{m,n}=\int_0^{2\pi} \cos^m(x)\sin^n(x)dx
$$
we have $C_{m,n}=0$ if at least one of $m$ and $n$ is odd. 

\noindent  \textcircled{3} There is a $c>0$ such that $1-\cos(x)\ge cx^2$ for all $x\in [-3\pi/2,3\pi/2]$.  
\begin{proof}
Consider function 
$$
f(x)=\left\{
\begin{aligned}
\frac{1-\cos(x)}{x^2}, \  &x\not=0\\
\frac{1}{2} \ \ \ \ \ \ \ , \    &x=0. 
\end{aligned}
\right.
$$
It is easy to see by its Taylor expansion that $f(x)\in C^\infty(\RR)$. And we have $f(x)>0$ for all $x\in [-3\pi/2,3\pi/2]$. Thus let $c=\min_{x\in [-3\pi/2,3\pi/2]}\{f(x)\}>0$, we have
$$
1-\cos(x)\ge cx^2, \ \forall x\in [-3\pi/2,3\pi/2].
$$
Actually, one can easily evaluate that $c=\frac{4}{9\pi^2}$. 
\end{proof}

\noindent  \textcircled{4} There is some $c_1>0$ such that within $[-c_1,c_1]$,
$$
e^x\le 1+x+x^2.
$$
\begin{proof}
Again consider function 
$$
g(x)=\left\{
\begin{aligned}
\frac{e^x-1-x}{x^2}, \  &x\not=0\\
\frac{1}{2} \ \ \ \ \ \ \ , \    &x=0. 
\end{aligned}
\right.
$$
It is easy to see by its Taylor expansion that $g(x)\in C^\infty(\RR)$. So by continuity of $g$  around 0, there must be such a $c_1>0$. Actually one can easily check that $c_1=1$ satisfies the inequality here.  
\end{proof}

\noindent  \textcircled{5} For any $x>0$, $\log(1+x)\le x$. 
\begin{proof}
Note that $x=\log(x+1)$ and that
$$
\frac{dx}{dx}=1=\frac{d\log(x+1)}{dx}=\frac{1}{1+x}
$$
at $x=0$, $x$ is linear while $\log(1+x)$ is concave. 
\end{proof}

\noindent  \textcircled{6} With \textcircled{2}, we can also have that for any $n\in \ZZ^+$, integers $k_1,k_2,\cdots, k_n\ge 0$ and any $a_1,a_2,\cdots, a_n\in R$, suppose
$$
K=\sum_{i=1}^n k_n
$$ 
is a odd number. We always have
$$
\int_{-\pi}^{\pi} \prod_{i=1}^n\cos^{k_i}(\theta-a_i)d\theta=0. 
$$
\begin{proof}
For each $i$, note that 
$$
\cos(\theta-a_i)=\cos(a_i)\cos(\theta)+\sin(a_i)\sin(\theta). 
$$
Thus we can expand $\prod_{i=1}^n\cos^{k_i}(\theta-a_i)$ as a binomial of $\cos(\theta)$ and $\sin(\theta)$ and get
\beq
\label{6 summation}
\prod_{i=1}^n\cos^{k_i}(\theta-a_i)=\sum_{j=0}^K c_j \cos^j(\theta)\sin^{K-j}(\theta). 
\eeq
Then look at each term in the summation on the right hand side of \eqref{6 summation}. For any $j$ one can immediately have one between $j$ and $K-j$ must be odd since $K$ is odd. Thus \textcircled{2} finishes the proof. 
\end{proof}

\subsection{Returning Probability to 0}\label{sec:returnzeroapp}

In this section we prove Theorem \ref{Main Theorem}. It is well known (e.g. \cite{randomwalkbook}) that 
$$
P(\tau_{d,0}<\infty)=1-G_d^{-1}(0)
$$
where $G_d(\cdot)$ is the $d$ dimensional Green's function. I.e., 
\beq
\label{G(x)}
G_d(x)=\left(\frac{1}{2\pi}\right)^d\int_{[-\pi,\pi]^d}\frac{1}{1-\phi_d(\theta)}e^{-ix\cdot\theta}d\theta,
\eeq
where 
$$
\phi_d(\theta)=\frac{1}{d}\sum_{i=1}^d \cos(\theta_i)
$$
is the character function of $d$ dimensional simple random walk. See Section 4.1 of \cite{randomwalkbook} for details. Since
$$
p_d=P(\tau_{d,0}<\infty)=1-G_d^{-1}(0)=\frac{G_d(0)-1}{G_d(0)},
$$
to prove the asymptotics we want, it suffices to show that 
\beq
\label{limit of G(0)}
\lim_{d\to\infty} 2d[G_d(0)-1]=1.
\eeq 
To show \eqref{limit of G(0)}, noting that 
\beq
\label{G_d(0)}
G_d(0)=\left(\frac{1}{2\pi}\right)^d\int_{[-\pi,\pi]^d}\displaystyle{\frac{1}{1-\frac{1}{d}\sum_{i=1}^d \cos(\theta_i)}}d\theta_1d\theta_2\cdots d\theta_d
\eeq
we will concentrate on studying this integration. Since 
$$
\frac{1}{1-x}=1+x+x^2+x^3+\frac{x^4}{1-x}
$$
for all $x\not=1$, we have 
\beq
\label{G(0) decomposition}
\begin{aligned}
G_d(0)=1&+\left(\frac{1}{2\pi}\right)^d\int_{[-\pi,\pi]^d}\left( \frac{\sum_{i=1}^d \cos(\theta_i)}{d}\right)d\theta_1d\theta_2\cdots d\theta_d\\
&+\left(\frac{1}{2\pi}\right)^d\int_{[-\pi,\pi]^d}\left( \frac{\sum_{i=1}^d \cos(\theta_i)}{d}\right)^2d\theta_1d\theta_2\cdots d\theta_d\\
&+\left(\frac{1}{2\pi}\right)^d\int_{[-\pi,\pi]^d}\left( \frac{\sum_{i=1}^d \cos(\theta_i)}{d}\right)^3d\theta_1d\theta_2\cdots d\theta_d\\
&+\left(\frac{1}{2\pi}\right)^d\int_{[-\pi,\pi]^d}\frac{\left( \frac{\sum_{i=1}^d \cos(\theta_i)}{d}\right)^4}{1-\left( \frac{\sum_{i=1}^d \cos(\theta_i)}{d}\right)}d\theta_1d\theta_2\cdots d\theta_d.
\end{aligned}
\eeq
For the first integration in \eqref{G(0) decomposition}, we have for any $i=1,2,\cdots, d$, by \textcircled{2}
$$
\left(\frac{1}{2\pi}\right)^d\int_{[-\pi,\pi]^d} \cos(\theta_i) d\theta_1d\theta_2\cdots d\theta_d=\frac{1}{2\pi}\int_{-\pi}^\pi \cos(\theta_i) d\theta_i=0. 
$$
So
\beq
\label{integration 1}
\left(\frac{1}{2\pi}\right)^d\int_{[-\pi,\pi]^d}\left( \frac{\sum_{i=1}^d \cos(\theta_i)}{d}\right)d\theta_1d\theta_2\cdots d\theta_d=0.
\eeq
Then for the second integration in \eqref{G(0) decomposition}, we have
\begin{align*}
&\left(\frac{1}{2\pi}\right)^d\int_{[-\pi,\pi]^d}\left( \frac{\sum_{i=1}^d \cos(\theta_i)}{d}\right)^2d\theta_1d\theta_2\cdots d\theta_d\\
& \hspace{0.4 in} =\frac{1}{d^2}\left(\frac{1}{2\pi}\sum_{i=1}^d\int_{-\pi}^{\pi} \cos^2(\theta_i)d\theta_i+\left(\frac{1}{2\pi}\right)^2 \sum_{i\not=j} \int_{[-\pi,\pi]^2} \cos(\theta_i)\cos(\theta_j)d\theta_id\theta_j\right).
\end{align*}
Note that for any $i\not=j$, 
$$
\int_{[-\pi,\pi]^2} \cos(\theta_i)\cos(\theta_j)d\theta_id\theta_j=\left(\int_{-\pi}^\pi \cos(\theta_i) d\theta_i \right)\times \left(\int_{-\pi}^\pi \cos(\theta_j) d\theta_j \right)=0
$$
and that
$$
\int_{-\pi}^{\pi} \cos^2(\theta_i)d\theta_i=\int_{-\pi}^{\pi} \frac{\cos(2\theta_i)+1}{2}d\theta_i=\pi.
$$
So
\beq
\label{integration 2}
\left(\frac{1}{2\pi}\right)^d\int_{[-\pi,\pi]^d}\left( \frac{\sum_{i=1}^d \cos(\theta_i)}{d}\right)^2d\theta_1d\theta_2\cdots d\theta_d=\frac{1}{2\pi d^2}\sum_{i=1}^d \pi=\frac{1}{2d}. 
\eeq
Moreover, note that 
\begin{align*}
&\left(\frac{1}{2\pi}\right)^d\int_{[-\pi,\pi]^d}\left( \sum_{i=1}^d \cos(\theta_i)\right)^3d\theta_1d\theta_2\cdots d\theta_d\\
&\hspace{0.4 in} =\frac{1}{2\pi}\sum_{i=1}^d \int_{-\pi}^{\pi} \cos^3(\theta_i)d\theta_i+\left(\frac{1}{2\pi}\right)^2 \sum_{i\not=j}3\int_{[-\pi,\pi]^2}\cos(\theta_i)^2\cos(\theta_j)d\theta_id\theta_j\\
&\hspace{0.4 in} +\left(\frac{1}{2\pi}\right)^3  \sum_{i\not=j\not=k} \int_{[-\pi,\pi]^3}\cos(\theta_i)\cos(\theta_j)\cos(\theta_k)d\theta_id\theta_jd\theta_k\\
 &\hspace{0.4 in} =0. 
\end{align*}
Thus
\beq
\label{integration 3}
\left(\frac{1}{2\pi}\right)^d\int_{[-\pi,\pi]^d}\left( \frac{\sum_{i=1}^d \cos(\theta_i)}{d}\right)^3d\theta_1d\theta_2\cdots d\theta_d=0.
\eeq
Combining  \eqref{G(0) decomposition} and \eqref{integration 1}-\eqref{integration 3} , we have
\beq
\label{G(0) new}
G_d(0)=1+\frac{1}{2d}+\left(\frac{1}{2\pi}\right)^d\int_{[-\pi,\pi]^d}\frac{\left( \frac{\sum_{i=1}^d \cos(\theta_i)}{d}\right)^4}{1-\left( \frac{\sum_{i=1}^d \cos(\theta_i)}{d}\right)}d\theta_1d\theta_2\cdots d\theta_d.
\eeq
Thus to have the desired asymptotic, it is sufficient to show that 
\beq
\label{o(d^-1)}
\mathcal{E}_d=\left(\frac{1}{2\pi}\right)^d\int_{[-\pi,\pi]^d}\frac{\left( \frac{\sum_{i=1}^d \cos(\theta_i)}{d}\right)^4}{1-\left( \frac{\sum_{i=1}^d \cos(\theta_i)}{d}\right)}d\theta_1d\theta_2\cdots d\theta_d=o(d^{-1}). 
\eeq
To show \ref{o(d^-1)}, we do not know how to do it with only calculus. However, since there are terms in the integration which look similar to sampled mean of random variables, we will construct a probability model and use large deviation technique to solve the problem. 

\

\noindent Let $X_1,X_2,\cdots, X_d$ be i.i.d. uniform random variables on $[-\pi,\pi]$, and let $Y_i=\cos(X_i)$, $i=1,2,\cdots,d$, which are also i.i.d random variables. Moreover, one can immediately check that $Y_i\in [-1,1]$ and
$$
E[Y_i]=\frac{1}{2\pi}\int_{-\pi}^\pi \cos(x) dx=0,
$$ 
and
$$
\var(Y_i)=\frac{1}{2\pi}\int_{-\pi}^\pi \cos^2(x) dx=\frac{1}{2}. 
$$
Then the sampled mean 
$$
\bar Y_d=\frac{Y_1+Y_2+\cdots+Y_d}{d}\in [-1,1]
$$
and let $Z_d$ be the nonnegative random variable
$$
Z_d=\left\{
\begin{aligned}
&\frac{\bar Y_d^4}{1-\bar Y_d}, \ \ \bar Y_d<1\\
&0, \hspace{0.5 in} \bar Y_d=1.
\end{aligned}
\right.
$$ 
Then according to our construction and the definition of $\mathcal{E}_d$, we have 
\beq
\label{expectation}
\mathcal{E}_d=E[Z_d]. 
\eeq
To bound \eqref{expectation}, define the event $A_d=\{|\bar Y_d|\le d^{-0.4}\}$, then for any $d\ge 6$, 
$$
E[Z_d]\le \frac{d^{-1.6}}{1-d^{-0.4}}P(A_d)+E[Z_d\ind_{A^c_d}]\le 2d^{-1.6}+E[Z_d\ind_{A^c_d}].
$$
Then define the event 
$$
B_d=\left\{\sqrt{X_1^2+X_2^2+\cdots+X_d^2}\le\frac{1}{d}\right\}.
$$
We can further have 
\beq
\label{expectation decomposition_2}
\begin{aligned}
E[Z_d]&\le 2d^{-1.6}+E[Z_d\ind_{A^c_d\cap B_d^c}]+E[Z_d\ind_{A^c_d\cap B_d}]\\
&\le 2d^{-1.6}+P(A_d^c)\max_{\omega\in B_d^c} \{Z_d(\omega)\}+E[Z_d\ind_{B_d}]. 
\end{aligned}
\eeq
Note that the first term in \eqref{expectation decomposition_2} is already of order $o(d^{-1})$. In order to control the second term, we first have that whenever $\bar Y_d\not=0$,
$$
Z_d=\frac{\bar Y_d^4}{1-\bar Y_d}\le \frac{1}{1-\bar Y_d}=\frac{d}{\sum_{i=1}^d1-\cos(X_i)}.
$$
Then according to \textcircled{3}, we have
\beq
\label{upper bound Z_d}
Z_d\le \frac{d}{c\sum_{i=1}^d X_i^2}
\eeq
which implies that
\beq
\label{upper bound Z_d 1}
\max_{\omega\in B_d^c} \{Z_d(\omega)\}\le \frac{d^3}{c}. 
\eeq
On the other hand, noting that we have $d^{-0.4}$ which is of a higher order than the central limit scaling $d^{-0.5}$. Cram\'{e}r's Theorem can be used to give an exponentially small term on $P(A_d^c)$. For the self-completeness of this paper we also include details of the large deviation argument used here. Note that 
$$
P(A_d^c)=P(\bar Y_d>d^{-0.4})+P(\bar Y_d<-d^{-0.4})
$$
and that $E[\bar Y_d]=E[-\bar Y_d]=0$, and that $\var(\bar Y_d)=\var(-\bar Y_d)=1/2$, we can without loss of generality concentrate on controlling the first probability. The calculations we have below will work exactly the same on the second one if we take $-\bar Y_d$ as our random variable. Note that for any $\theta>0$
$$
E[\exp(\bar Y_d \theta)]\le e^\theta<\infty. 
$$
And by Markov inequality, 
$$
P(\bar Y_d>d^{-0.4})\exp(d^{-0.4}\theta)\le E[\exp(\bar Y_d \theta)]
$$
for any $\theta>0$. Noting that $\bar Y_d$ is the sampled mean of i.i.d. sequence $Y_1,\cdots, Y_d$, one can rewrite the inequality above as
$$
P(\bar Y_d>d^{-0.4})\le \left(\frac{E\big[\exp(Y_1\frac{\theta}{d})\big]}{\exp\big(d^{-0.4}\frac{\theta}{d}\big)}\right)^d.
$$
Let $\theta_1=\theta/d$, then the inequality above can be further simplified as 
\beq
\label{large deviation 1}
P(\bar Y_d>d^{-0.4})\le\exp\left[d\Big(-d^{-0.4}\theta_1+\log\big(E[\exp(Y_1\theta_1)]\big) \Big) \right]
\eeq
which holds true for all $d$ and any $\theta_1>0$. Now let $\theta_1=d^{-0.4}$. Then we have that for all $d\ge 2$
$$
|Y_1\theta_1|\le \theta_1=d^{-0.4}<1=c_1
$$
which by \textcircled{4} implies 
$$
E[\exp(Y_1\theta_1)]\le E[1+Y_1\theta_1+(Y_1\theta_1)^2]=1+\frac{\theta_1^2}{2}. 
$$
Then by \textcircled{5} we can further have 
\beq
\label{large deviation 2}
P(\bar Y_d>d^{-0.4})\le\exp\left[d\left(-d^{-0.4}\theta_1+\frac{\theta_1^2}{2}\right) \right]=\exp(-d^{0.2}/2). 
\eeq
Combining \eqref{upper bound Z_d 1} and \eqref{large deviation 2}, 
\beq
\label{second term}
P(A_d^c)\max_{\omega\in B_d^c} \{Z_d(\omega)\}\le \frac{1}{c} d^3\exp(-d^{0.2}/2)=o(d^{-1}). 
\eeq
Finally, for the last term in \eqref{expectation decomposition_2}, noting that the upper bound of $Z_d$ in \eqref{upper bound Z_d} and the definition of $B_d$, we have 
\beq
\label{third term 1}
E[Z_d\ind_{B_d}]\le \left(\frac{1}{2\pi}\right)^d\int_{B_{2,d}(0,1/d)} \frac{d}{c\sum_{i=1}^d x_i^2} dx_1dx_2\cdots dx_d
\eeq
where $B_{2,d}(0,1/d)$ is the $L_2$ ball in $\RR^d$ centered at 0 with radius $1/d$. For the integral in \eqref{third term 1}, use the $d$ dimensional spherical coordinates
\begin{align*}
& x_1=r\cos(\theta_1)\\
& x_2=r\sin(\theta_1)\cos(\theta_2)\\
& x_3=r\sin(\theta_1)\sin(\theta_2)\cos(\theta_3)\\
&\cdot\\
&\cdot\\
&\cdot\\
&x_{d-1}=r\sin(\theta_1)\sin(\theta_2)\cdots \sin(\theta_{d-2})\cos(\theta_{d-1})\\
&x_{d-1}=r\sin(\theta_1)\sin(\theta_2)\cdots \sin(\theta_{d-2})\sin(\theta_{d-1})
\end{align*}
where $r\ge 0$, $\theta_i\in[0,\pi]$ for $i=1,2,\cdots,d-2$, and $\theta_{d-1}\in[0,2\pi]$. Then we have
\beq
\label{third term 2}
\begin{aligned}
&\left(\frac{1}{2\pi}\right)^d\int_{B_{2,d}(0,1/d)} \frac{d}{c\sum_{i=1}^d x_i^2} dx_1dx_2\cdots dx_d\\
&=\frac{d}{c(2\pi)^d}\int_{(0,1/d]\times[0,\pi]^{d-2}\times[0,2\pi]} r^{d-3}\prod_{i=1}^{d-2}\sin^{n-1-i}(\theta_i) dr d\theta_1d\theta_2\cdots d\theta_{d-1}\\
&\le \frac{d}{c(2\pi)^d}\int_{(0,1/d]\times[0,\pi]^{d-2}\times[0,2\pi]} r^{d-3} dr d\theta_1d\theta_2\cdots d\theta_{d-1}\\
&\le \frac{d}{c 2^{d-1}} \int_0^{1/d} r^{d-3} dr=\frac{1}{c 2^{d-1}} \cdot \frac{d}{d-2}\cdot d^{2-d}=o(d^{-1}). 
\end{aligned}
\eeq
Combining \eqref{third term 1} and \eqref{third term 2}, we have 
\beq
\label{third term}
E[Z_d\ind_{B_d}]\le\frac{1}{c 2^{d-1}} \cdot \frac{d}{d-2}\cdot d^{2-d}=o(d^{-1}). 
\eeq
Thus in \eqref{second term} and \eqref{third term}, we have shown that all terms in \eqref{expectation decomposition_2} is of order $o(d^{-1})$ and the proof of Theorem \ref{Main Theorem} is complete. 
\qed

\subsection{Returning Probability to the Diagonal Line}

In this section we prove Theorem \ref{Theorem diagonal}. Recalling that 
$$
\hat X_{d-1,n}=\left(X_{d,n}^1-X_{d,n}^2,X_{d,n}^2-X_{d,n}^3,\cdots, X_{d,n}^d-X_{d,n}^{d-1}\right)
$$
we have $X_{d,n}\in l_d$ if and only if $X_{d,n}^1=X_{d,n}^2=\cdots=X_{d,n}^d$, which in turns is equivalent to $\hat X_{d-1,n}=0$. And for the new process $\hat X_{d-1,n}$, one can easily check that it also forms a $d-1$ dimensional random walk with transition probability 
\begin{align*}
&P(\hat X_{d-1,1}=e_{d-1,1})=\frac{1}{2d}\\
&P(\hat X_{d-1,1}=-e_{d-1,1})=\frac{1}{2d}\\
&P(\hat X_{d-1,1}=e_{d-1,1}-e_{d-1,2})=\frac{1}{2d}\\
&P(\hat X_{d-1,1}=-e_{d-1,1}+e_{d-1,2})=\frac{1}{2d}\\
&P(\hat X_{d-1,1}=e_{d-1,2}-e_{d-1,3})=\frac{1}{2d}\\
&P(\hat X_{d-1,1}=-e_{d-1,2}+e_{d-1,3})=\frac{1}{2d}\\
&\cdot\\
&\cdot\\
&\cdot\\
&P(\hat X_{d-1,1}=e_{d-1,d-2}-e_{d-1,d-1})=\frac{1}{2d}\\
&P(\hat X_{d-1,1}=-e_{d-1,d-2}+e_{d-1,d-1})=\frac{1}{2d}\\
&P(\hat X_{d-1,1}=e_{d-1,d-1})=\frac{1}{2d}\\
&P(\hat X_{d-1,1}=-e_{d-1,d-1})=\frac{1}{2d}
\end{align*}
so that $\hat X_{d-1,n}$ also forms a finite range symmetric random walk. Moreover, the characteristic function of $\hat X_{d-1,n}$ is given by 
\beq
\hat\phi_{d-1}(\theta)=\frac{1}{d}\left(\cos(\theta_1)+\sum_{i=1}^{d-2}\cos(\theta_{i+1}-\theta_{i})+\cos(\theta_{d-1}) \right).
\eeq
And we also have 
$$
\hat \tau_{d-1,0}=\inf\{n>1: \hat X_{d-1,n}=0\}
$$
together with
$$
P(\hat \tau_{d-1,0}<\infty)=1-\hat G_{d-1}^{-1}(0)
$$
where $\hat G_{d-1}(\cdot)$ is the Green's function for $\hat X_{d-1,n}$. I.e., 
\beq
\label{hat G(0)}
\hat G_{d-1}(0)=\left(\frac{1}{2\pi}\right)^{d-1}\int_{[-\pi,\pi]^{d-1}}\frac{1}{1-\hat\phi_{d-1}(\theta)}d\theta. 
\eeq
Then again we only need to show that 
\beq
\label{limit hat G(0)}
\lim_{d\to\infty}2d[\hat G_{d-1}(0)-1]=1. 
\eeq
Moreover, using exactly the same embedded random walk argument as in Lemma 1 of \cite{Song96} on $X_{d,n}$ and $\tau_{d,l_d}$, one can immediately have $P_{d+1}\le P_{d}$, which is also equivalent to $\hat G_{d}(0)\le \hat G_{d-1}(0)$. So in order to show \eqref{limit hat G(0)}, we can without loss of generality concentrate on even $d$'s. 

\

\noindent Then again we have 
\beq
\label{hat G(0) decomposition}
 \begin{aligned}
 \hat G_{d-1}(0)=1&+\left(\frac{1}{2\pi}\right)^{d-1}\int_{[-\pi,\pi]^{d-1}}\hat\phi_{d-1}(\theta)d\theta\\
&+\left(\frac{1}{2\pi}\right)^{d-1}\int_{[-\pi,\pi]^{d-1}}\hat\phi_{d-1}^2(\theta)d\theta\\
&+\left(\frac{1}{2\pi}\right)^{d-1}\int_{[-\pi,\pi]^{d-1}}\hat\phi_{d-1}^3(\theta)d\theta\\
&+\left(\frac{1}{2\pi}\right)^{d-1}\int_{[-\pi,\pi]^{d-1}}\frac{\hat\phi_{d-1}^4(\theta)}{1-\hat\phi_{d-1}(\theta)}d\theta.
 \end{aligned}
\eeq
Note that by \textcircled{1} and \textcircled{2}, for any $i=1,\cdots, d-2$, 
$$
 \begin{aligned}
&\left(\frac{1}{2\pi}\right)^{d-1}\int_{[-\pi,\pi]^{d-1}} \cos(\theta_{i+1}-\theta_{i}) d\theta\\
& \hspace{0.2 in} =\left(\frac{1}{2\pi}\right)^2 \int_{-\pi}^{\pi}  \int_{-\pi}^{\pi}\cos(\theta_{i+1}-\theta_{i}) d\theta_{i+1} d\theta_{i}\\
& \hspace{0.2 in} =\left(\frac{1}{2\pi}\right)^2 \int_{-\pi}^{\pi}  \int_{-\pi-\theta_i}^{\pi-\theta_i}\cos(\theta) d\theta d\theta_{i}\\
& \hspace{0.2 in} =\left(\frac{1}{2\pi}\right)^2 \int_{-\pi}^{\pi}  0 d\theta_{i}=0,
 \end{aligned}
$$
$$
 \begin{aligned}
&\left(\frac{1}{2\pi}\right)^{d-1}\int_{[-\pi,\pi]^{d-1}} \cos^2(\theta_{i+1}-\theta_{i}) d\theta\\
& \hspace{0.2 in} =\left(\frac{1}{2\pi}\right)^2 \int_{-\pi}^{\pi}  \int_{-\pi}^{\pi}\cos^2(\theta_{i+1}-\theta_{i}) d\theta_{i+1} d\theta_{i}\\
& \hspace{0.2 in} =\left(\frac{1}{2\pi}\right)^2 \int_{-\pi}^{\pi}  \int_{-\pi-\theta_i}^{\pi-\theta_i}\cos^2(\theta) d\theta d\theta_{i}\\
& \hspace{0.2 in} =\left(\frac{1}{2\pi}\right)^2 \int_{-\pi}^{\pi}  \pi d\theta_{i}=\frac{1}{2},
 \end{aligned}
$$
and
$$
 \begin{aligned}
&\left(\frac{1}{2\pi}\right)^{d-1}\int_{[-\pi,\pi]^{d-1}} \cos^3(\theta_{i+1}-\theta_{i}) d\theta\\
& \hspace{0.2 in} =\left(\frac{1}{2\pi}\right)^2 \int_{-\pi}^{\pi}  \int_{-\pi}^{\pi}\cos^3(\theta_{i+1}-\theta_{i}) d\theta_{i+1} d\theta_{i}\\
& \hspace{0.2 in} =\left(\frac{1}{2\pi}\right)^2 \int_{-\pi}^{\pi}  \int_{-\pi-\theta_i}^{\pi-\theta_i}\cos^3(\theta) d\theta d\theta_{i}\\
& \hspace{0.2 in} =\left(\frac{1}{2\pi}\right)^2 \int_{-\pi}^{\pi}  0 d\theta_{i}=0. 
 \end{aligned}
$$
It is easy to see the first 3 integration terms will be exactly the same as in Section \ref{sec:returnzeroapp}. I.e.,
$$
\left(\frac{1}{2\pi}\right)^{d-1}\int_{[-\pi,\pi]^{d-1}}\hat\phi_{d-1}(\theta)d\theta=\left(\frac{1}{2\pi}\right)^{d-1}\int_{[-\pi,\pi]^{d-1}}\hat\phi_{d-1}^3(\theta)d\theta=0
$$
and 
$$
\left(\frac{1}{2\pi}\right)^{d-1}\int_{[-\pi,\pi]^{d-1}}\hat\phi_{d-1}^2(\theta)d\theta=\frac{1}{2d}.
$$
Thus, we again have 
\beq
 \hat G_{d-1}(0)=1+\frac{1}{2d}+\left(\frac{1}{2\pi}\right)^{d-1}\int_{[-\pi,\pi]^{d-1}}\frac{\hat\phi_{d-1}^4(\theta)}{1-\hat\phi_{d-1}(\theta)}d\theta.
\eeq
And we only need to show that for sufficiently large even $d$
\beq
\label{o(d^-1)2}
\mathcal{\hat{E}}_{d-1}=\left(\frac{1}{2\pi}\right)^{d-1}\int_{[-\pi,\pi]^{d-1}}\frac{\hat\phi_{d-1}^4(\theta)}{1-\hat\phi_{d-1}(\theta)}d\theta=o(d^{-1}).
\eeq
To show \eqref{o(d^-1)2}, we again rewrite the integral above into the expectation of some function of a sequence of i.i.d. random variables. Let $\hat X_1,\hat X_2,\cdots, \hat X_{d-1}$ be i.i.d. uniform random variables on $[-\pi,\pi]$, we can define 
$$
\hat Y_{d-1}=\frac{1}{d}\left(\cos(\hat X_1)+\sum_{i=1}^{d-2}\cos(\hat X_{i+1}-\hat X_{i})+\cos(\hat X_{d-1}) \right)\in [-1,1]
$$
and 
$$
\hat Z_{d-1}=\left\{
\begin{aligned}
&\frac{\hat Y_{d-1}^4}{1-\hat Y_{d-1}}, \ \  \hat Y_{d-1}<1\\
&0, \hspace{0.65 in} \hat Y_{d-1}=1.
\end{aligned}
\right.
$$ 
Then according to our construction and the definition of $\mathcal{\hat E}_d$, we have 
\beq
\label{expectation hat}
\mathcal{\hat E}_{d-1}=E[\hat Z_{d-1}]. 
\eeq
Again let event $\hat A_{d-1}=\{|\hat Y_{d-1}|\le d^{-0.4}\}$, then for any $d\ge 6$, 
$$
E[\hat Z_{d-1}]\le \frac{d^{-1.6}}{1-d^{-0.4}}P(\hat A_{d-1})+E[\hat Z_{d-1}\ind_{\hat A^c_{d-1}}]\le 2d^{-1.6}+E[\hat Z_{d-1}\ind_{\hat A^c_{d-1}}].
$$
Then let 
$$
\hat B_{d-1}=\left\{\sqrt{\hat X_1^2+\hat X_2^2+\cdots+\hat X_{d-1}^2}\le\frac{1}{d}\right\}.
$$
We can further have 
\beq
\label{expectation decomposition}
\begin{aligned}
E[\hat Z_{d-1}]&\le 2d^{-1.6}+E[\hat Z_{d-1}\ind_{\hat A^c_{d-1}\cap \hat B_{d-1}^c}]+E[\hat Z_{d-1}\ind_{A^c_{d-1}\cap \hat B_{d-1}}]\\
&\le 2d^{-1.6}+P(\hat A^c_{d-1})\max_{\omega\in \hat B_{d-1}^c} \{\hat Z_{d-1}(\omega)\}+E[\hat Z_{d-1}\ind_{\hat B_{d-1}}]. 
\end{aligned}
\eeq
To control the third term in \eqref{expectation decomposition}, note that for any $d\ge 3$, and any $i=1,2,\cdots, d-2$, within the event $\hat B_{d-1}$,
$$
|X_i-X_j|\le \frac{2}{d}\le \pi. 
$$
Thus within the event $\hat B_{d-1}\cap\{\hat Y_{d-1}\not=1\}$, we have by \textcircled{3}
\beq
\label{hat term 3 1}
\hat Z_{d-1}\le \frac{1}{1-\hat Y_{d-1}}\le \frac{d}{c\left(\hat X_1^2+\sum_{i=1}^{d-2} |\hat X_{i+1}-\hat X_i|^2+\hat X_{d-1}^2\right)}.
\eeq
Moreover, for any $(x_1,x_2,\cdots, x_{d-1})\in \RR^d$, we have 
$$
x_1^2+\sum_{i=1}^{d-2} |x_{i+1}-x_i|^2+x_{d-1}^2\ge \sigma(d-1,1)^2 \left(\sum_{i=1}^{d-1}x_i^2  \right)
$$
where $\sigma(d-1,1)$ is the smallest singular value of $d-1$ by $d-1$ Jordan block with $\lambda=1$. In \cite{singular_value} it has been proved that 
$$
\sigma(d-1,1)\ge \frac{1}{d-1}\ge \frac{1}{d}.
$$
Thus we have
\beq
\label{hat term 3 2}
x_1^2+\sum_{i=1}^{d-2} |x_{i+1}-x_i|^2+x_{d-1}^2\ge d^{-2} \left(\sum_{i=1}^{d-1}x_i^2  \right).
\eeq
Combining \eqref{hat term 3 1} and \eqref{hat term 3 2} gives us 
\beq
\label{hat term 3 3}
\hat Z_{d-1}\le  \frac{d^3}{c\sum_{i=1}^{d-1} \hat X_i^2}
\eeq
which implies that
\beq
\label{hat term 3 3}
E[\hat Z_{d-1}\ind_{\hat B_{d-1}}]\le \left(\frac{1}{2\pi}\right)^{d-1}\int_{B_{2,d-1}(0,1/d)} \frac{d^3}{c\sum_{i=1}^d x_i^2} dx_1dx_2\cdots dx_{d-1}.
\eeq
Using exactly the same argument of spherical coordinates, we have 
\beq
\label{hat term 3}
E[\hat Z_{d-1}\ind_{\hat B_{d-1}}]\le \frac{1}{c 2^{d-2}}\frac{d^3}{d-2} d^{2-d}=o(d^{-1}). 
\eeq
Then for the second term $P(\hat A^c_{d-1})\max_{\omega\in \hat B_{d-1}^c} \{\hat Z_{d-1}(\omega)\}$, we first control the probability $P(\hat A^c_{d-1})$ for sufficiently large even number $d=2n$. Note that 
$$
\hat Y_{d-1}\le \frac{2}{d}+\frac{1}{d}\sum_{i=1}^{d-2}\cos(\hat X_{i+1}-\hat X_{i})
$$
and that 
$$
\hat Y_{d-1}\ge -\frac{2}{d}+\frac{1}{d}\sum_{i=1}^{d-2}\cos(\hat X_{i+1}-\hat X_{i}).
$$
So we have for sufficiently large even number $d=2n$
$$
P(\hat Y_{d-1}\ge d^{-0.4})\le P\left(\frac{1}{d}\sum_{i=1}^{d-2}\cos(\hat X_{i+1}-\hat X_{i})\ge \frac{d^{-0.4}}{2}\right)
$$
and 
$$
P(\hat Y_{d-1}\le -d^{-0.4})\le P\left(\frac{1}{d}\sum_{i=1}^{d-2}\cos(\hat X_{i+1}-\hat X_{i})\le \frac{d^{-0.4}}{2}\right).
$$
Moreover note that for $d=2n$ we have 
$$
\frac{1}{d}\sum_{i=1}^{d-2}\cos(\hat X_{i+1}-\hat X_{i})= \frac{n-1}{2n}(\hat Y_{1,d-1}+\hat Y_{2,d-1})
$$
where
$$
\hat Y_{1,d-1}=\frac{1}{n-1}\sum_{i=1}^{n-1}\cos(\hat X_{2i+1}-\hat X_{2i})
$$
and
$$
\hat Y_{2,d-1}=\frac{1}{n-1}\sum_{i=1}^{n-1}\cos(\hat X_{2i}-\hat X_{2i+1}).
$$
Noting that $\hat Y_{1,d-1}$ and $\hat Y_{2,d-1}$ are again sampled means of i.i.d. random variables with expectation 0 and variance $1/2$. Although now we have $\hat Y_{1,d-1}$ and $\hat Y_{2,d-1}$ are correlated, we can still have the upper bound
$$
P\left(\frac{1}{d}\sum_{i=1}^{d-2}\cos(\hat X_{i+1}-\hat X_{i})\ge \frac{d^{-0.4}}{2}\right)\le P(\hat Y_{1,d-1}\ge \frac{d^{-0.4}}{2})+P(\hat Y_{2,d-1}\ge \frac{d^{-0.4}}{2})
$$
and
$$
P\left(\frac{1}{d}\sum_{i=1}^{d-2}\cos(\hat X_{i+1}-\hat X_{i})\le -\frac{d^{-0.4}}{2}\right)\le P(\hat Y_{1,d-1}\le -\frac{d^{-0.4}}{2})+P(\hat Y_{2,d-1}\le -\frac{d^{-0.4}}{2}).
$$
Apply Cram\'{e}r's Theorem on $\hat Y_{1,d-1}$ and $\hat Y_{2,d-1}$, we again have that there is some $u,U\in (0,\infty)$ (actually we can use $u=1/16$ and $U=2$) such that
\beq
\label{hat term 2 1}
P(\hat A^c_{d-1})\le U\exp(-u d^{0.2}).
\eeq
Lastly for $\max_{\omega\in \hat B_{d-1}^c} \{\hat Z_{d-1}(\omega)\}$, note that the range of $\hat X_{i+1}-\hat X_{i}$ is $[-2\pi,2\pi]$ which is no longer a subset of $[-3\pi/2,3\pi/2]$, we will not be able to use \textcircled{3} directly to find an upper bound. to overcome this issue, we have the following lemma:
\begin{lemma}
\label{Lemma maximum}
For any $d$, consider the following two subsets of $\RR^{d-1}$:
$$
D_1(d)=[-\pi,\pi]^{d-1}\cap \left\{\frac{1-\cos(x_1)}{d}+\sum_{i=1}^{d-2}\frac{1-\cos(x_{i+1}-x_i)}{d}+\frac{1-\cos(x_{d-1})}{d}\le d^{-7}\right\}
$$
and
$$
D_2(d)=\left\{(x_1,\cdots, x_{d-1}): \ |x_i|\le \frac{i}{\sqrt{c}d^3}, \ \forall i=1,2,\cdots, d-1\right\}
$$
where $c$ is the constant in \textcircled{3}. Then there is some $d_0<\infty$ such that for all $d\ge d_0$, $D_1(d)\subseteq D_2(d)$. 
\end{lemma}

\begin{proof}
Let $d_0$ be a positive integer such that $\sqrt{c} d_0^2>1$. Then for any $d\ge d_0$ and any $(x_1,\cdots, x_{d-1})\in D_1(d)$. By the definition of $D_1(d)$ and the fact that $x_1\in [-\pi, \pi]$,  we must have 
$$
\frac{cx_1^2}{d}\le \frac{1-\cos(x_1)}{d}\le d^{-7}
$$
which implies that 
\beq
\label{x_1}
|x_1|\le \frac{1}{\sqrt{c}d^3}.
\eeq
Now suppose there is a $(x_1,\cdots, x_{d-1})\in D_1(d)\cap D_2(d)^c$. Let $k=\inf\{i:  \ |x_i|> \frac{i}{\sqrt{c}d^3}\}$. Then \eqref{x_1} ensures that $k>1$. Then for $x_{k-1}$, 
$$
|x_{k-1}|\le  \frac{k-1}{\sqrt{c}d^3}\le  \frac{d}{\sqrt{c}d^3}\le \frac{1}{\sqrt{c} d_0^2}<1. 
$$
Thus we must have $|x_{k-1}-x_k|\le 3\pi/2$, which gives that 
$$
\frac{1-\cos(x_{k}-x_{k-1})}{d}\ge \frac{c}{d}|x_{k}-x_{k-1}|^2\ge \frac{c}{d} (|x_k|-|x_{k-1}|)^2>\frac{1}{d^7}. 
$$
And now we have a contradiction. 
\end{proof}
\noindent Moreover, we have that for any $(x_1,\cdots, x_{d-1})\in D_2(d)$,
$$
x_1^2+\cdots+x_{d-1}^2\le \frac{\sum_{i=1}^{d-1} i^2}{cd^6}=O(d^{-3})=o(d^{-2}).
$$
Thus there is another $d_1<\infty$ such that for all $d\ge d_1$,
\beq
D_1(d)\subset D_2(d)\subset B_{2,d}(0,1/d). 
\eeq
This means for any $d\ge d_1$, and any $(x_1,\cdots, x_{d-1})\in B_{2,d}(0,1/d)^c$, 
\beq
\label{hat term 2 2}
\frac{1}{1-\frac{1}{d}\left(\cos(x_1)+\sum_{i=1}^{d-2}\cos(x_{i+1}-x_i)+\cos(x_{d-1})\right)}\le d^7,
\eeq
which gives us 
\beq
\label{hat term 2 3}
\max_{\omega\in \hat B_{d-1}^c} \{\hat Z_{d-1}(\omega)\}\le d^7. 
\eeq
Thus combine \eqref{hat term 2 1} and \eqref{hat term 2 3} we finally have 
\beq
\label{hat term 2}
P(\hat A^c_{d-1})\max_{\omega\in \hat B_{d-1}^c} \{\hat Z_{d-1}(\omega)\} \le Ud^7\exp(-u d^{0.2})=o(d^{-1})
\eeq
and the proof of Theorem \ref{Theorem diagonal} is complete. 
\qed

\subsection{Proof of Theorem \ref{Theorem return d}}
With the asymptotic of the return of $\{\hat X_{d-1,n}\}_{n=0}^\infty$ obtained in the previous section, we will be able to use a similar but more complicate argument to show the same asymptotic for $\{\hat X_{d-1,n}\}_{n=0}^\infty$ to return to the set $D_{d-1}$. First using again exactly the same embedded random walk argument as in Lemma 1 of \cite{Song96}, it is easy to note that each time $\{\hat X_{d-1,n}\}_{n=0}^\infty$ returns to $D_{d-1}$ is also a time when $\{\hat X_{d-2,n}\}_{n=0}^\infty$, the embedded Markov chain tracking the changes of the first $d-2$ coordinates of  $\{\hat X_{d-1,n}\}_{n=0}^\infty$, which is also a $d-2$ dimensional version of the non simple random walk of interest, returns to $D_{d-2}$. This implies 
$$
\sup_{0\le i\le d-1}P\left(\inf_{0\le j\le d-1}\left\{T^{(i,j)}_{d-1}\right\}<\infty\right)\le \sup_{0\le i\le d-2}P\left(\inf_{0\le j\le d-2}\left\{T^{(i,j)}_{d-2}\right\}<\infty\right)
$$
and we can without loss of generality again concentrate on even numbers of $d$'s. Then for each $i$, one can immediately have 
$$
P\left(\inf_{0\le j\le d-1}\left\{T^{(i,j)}_{d-1}\right\}<\infty\right)\le \sum_{j=0}^{d-1} P\big(T^{(i,j)}_{d-1}<\infty \big). 
$$
Thus, in order to prove Theorem \ref{Theorem return d}, it is sufficient to show that for all sufficiently large even $d$'s, there is a $C<\infty$ such that for any $0\le i\le d-1$
\beq
\label{equation sharp 1}
 \sum_{j=0}^{d-1} P\big(T^{(i,j)}_{d-1}<\infty \big)<\frac{C}{d}. 
\eeq
Then for any $0\le i\not=j\le d-1$, by strong Markov property
$$
P\big(T^{(i,j)}_{d-1}<\infty \big)=\frac{\hat G_{d-1}(e_j-e_i)}{\hat G_{d-1}(0)}. 
$$
and 
$$
P\big(T^{(i,i)}_{d-1}<\infty \big)=\frac{\hat G_{d-1}(0)-1}{\hat G_{d-1}(0)}. 
$$
In Theorem \ref{Theorem diagonal}, we have already proved that $\hat G_{d-1}(0)=1+(2d)^{-1}+o(d^{-1})$. Thus now it is sufficient to show that for any $i$
\beq
\label{equation sharp 2}
 \sum_{j=0}^{d-1} \hat G_{d-1}(e_j-e_i)-1\le \frac{C}{d}. 
\eeq
For any $0\le i, j\le d-1$, we have 
\beq
\label{equation sharp 3}
\hat G_{d-1}(e_j-e_i)=\hat G_{d-1}(e_i-e_j)=\left(\frac{1}{2\pi}\right)^{d-1}\int_{[-\pi,\pi]^{d-1}}\frac{\cos(\theta_j-\theta_i)}{1-\hat\phi_{d-1}(\theta)}d\theta. 
\eeq
Thus, we will concentrate on controlling
$$
G_{d-1}^{(i,j)}=\left(\frac{1}{2\pi}\right)^{d-1}\int_{[-\pi,\pi]^{d-1}}\frac{\cos(\theta_j-\theta_i)}{1-\hat\phi_{d-1}(\theta)}d\theta
$$
with $0\le i<j\le d-1$. Using the same technique as in the proof of Theorem \ref{Theorem diagonal}, and noting that 
$$
\int_{[-\pi,\pi]^{d-1}}\cos(\theta_j-\theta_i) d\theta=0
$$
we first have 
\beq
\begin{aligned}
G_{d-1}^{(i,j)}=&\left(\frac{1}{2\pi}\right)^{d-1}\sum_{p=1}^5\left( \int_{[-\pi,\pi]^{d-1}}\cos(\theta_j-\theta_i)\hat\phi^p_{d-1}(\theta) d\theta\right) \\
&+\left(\frac{1}{2\pi}\right)^{d-1}\int_{[-\pi,\pi]^{d-1}}\frac{\cos(\theta_j-\theta_i)\hat\phi^6_{d-1}(\theta)}{1-\hat\phi_{d-1}(\theta)}d\theta
\end{aligned}
\eeq
and we call
\beq
\mathcal{\hat E}^{(i,j)}_d=\left(\frac{1}{2\pi}\right)^{d-1}\int_{[-\pi,\pi]^{d-1}}\frac{\cos(\theta_j-\theta_i)\hat\phi^6_{d-1}(\theta)}{1-\hat\phi_{d-1}(\theta)}d\theta
\eeq
to be the tail term. For any $0\le i\not=j\le d-1$, let $d(i,j)$ be their distance up to mod$(d)$. I.e.,
$$
d(i,j)=\min\{|j-i|, d-|j-i|\}\ge 1. 
$$
The reason we want to have the distance under mod$(d)$ is that our non simple random walk $\{\hat X_{d-1,n}\}_{n=0}^\infty$ has some ``periodic boundary condition" where we need one transition to move from $e_{d-1,d-1}$ to $e_{d-1,0}$. Then we have the following lemma which implies that for all but a finite number of $(i,j)$'s, the tail term $\mathcal{\hat E}^{(i,j)}_d$ is actually all we get for $G_{d-1}^{(i,j)}$.   
\begin{lemma}
\label{Lemma tail}
For any $k\in \ZZ^+$ and any $0\le i\not=j\le d-1$ such that $d(i,j)>k$, 
\beq
\label{equation tail 1}
 \int_{[-\pi,\pi]^{d-1}}\cos(\theta_j-\theta_i)\hat\phi^k_{d-1}(\theta) d\theta=0. 
\eeq
\end{lemma}
\begin{proof}
By symmetry we can without generality assume that $j>i$. Recalling that 
$$
\hat\phi_{d-1}(\theta)=\frac{1}{d}\left(\cos(\theta_1)+\sum_{i=1}^{d-2}\cos(\theta_{i+1}-\theta_{i})+\cos(\theta_{d-1}) \right),
$$
we have
$$
\hat\phi^k_{d-1}(\theta)=\frac{1}{d^k} \sum_{0\le i_1,i_2,\cdots, i_k\le d-1} \prod_{h=1}^k \cos(\theta_{i_h}-\theta_{i_h+1}),
$$
where we use the convention that $\theta_0=\theta_d=0$. For each term in the summation above, it is easy to see that there is some nonnegative integers $k_0,\cdots, k_{d-1}$ with $\sum_{h=0}^{d-1}k_h=k$ such that we can rewrite the term as 
\beq
\label{equation tail 2}
\prod_{h=0}^{d-1} \cos^{k_h}(\theta_h-\theta_{h+1}). 
\eeq
Thus, it is sufficient to show that for any nonnegative integers $k_0,\cdots, k_{d-1}$ with $\sum_{h=0}^{d-1}k_h=k$ 
\beq
\label{equation tail 3}
\int_{[-\pi,\pi]^{d-1}}\cos(\theta_j-\theta_i)\left(\prod_{h=0}^{d-1} \cos^{k_h}(\theta_h-\theta_{h+1})\right)d\theta=0. 
\eeq
First, if $i=0$ then we have $j>k$ and $d-j>k$. Thus we can separate the product in \eqref{equation tail 2} as 
$$
\prod_{h=0}^{d-1} \cos^{k_h}(\theta_h-\theta_{h+1})=\Pi[0:j-1]\cdot \Pi[j:d-1]
$$
where
$$
\Pi[0:j-1]=\prod_{h=0}^{j-1} \cos^{k_h}(\theta_h-\theta_{h+1}),  \ \Pi[j:d-1]=\prod_{h=j}^{d-1} \cos^{k_h}(\theta_h-\theta_{h+1}). 
$$
Thus $\Pi[0:j-1]$ is a product of $j$ terms while  $\Pi[j:d-1]$ is a product of $d-j$ terms. Note that 
\begin{align*}
&\cos(\theta_j)\left(\prod_{h=0}^{d-1} \cos^{k_h}(\theta_h-\theta_{h+1})\right)\\
&\hspace{0.2 in}=\cos(\theta_j)\cos^{k_{j-1}}(\theta_j-\theta_{j-1})\cos^{k_j}(\theta_j-\theta_{j+1})\left(\prod_{h\in \{0, \cdots, d-1\}\setminus\{j-1,j\}} \cos^{k_h}(\theta_h-\theta_{h+1})\right).
\end{align*}
If $k_{j-1}+k_{j}$ is an even number, integrate over $\theta_j$ and \textcircled{6} gives us \eqref{equation tail 3}. If $k_{j-1}+k_{j}$ is odd, without loss of generality we can assume $k_{j-1}$ is odd. Noting that 
$$
\sum_{h=0}^{j-1}k_h\le k<j, 
$$
by the pigeon hole principle we must have at least one of those $k_h$'s to be zero, which is even. Thus, let $h_0=\sup_{h\le j-1}\{k_h {\rm \ is  \ even}\}$. Then $h_0\in[0,j-2]$, where we use the standard convention that $\sup\{\O\}=-\infty$ and $\inf\{\O\}=\infty$. By definition $k_{h_0+1}$ is odd, and thus
\begin{align*}
&\cos(\theta_j)\left(\prod_{h=0}^{d-1} \cos^{k_h}(\theta_h-\theta_{h+1})\right)\\
=&\cos^{k_{h_0}}(\theta_{h_0+1}-\theta_{h_0})\cos^{k_{h_0+1}}(\theta_{h_0+1}-\theta_{h_0+2})\left(\cos(\theta_j)\prod_{h\in \{0, \cdots, d-1\}\setminus\{h_0,h_0+1\}} \cos^{k_h}(\theta_h-\theta_{h+1})\right).
\end{align*}
Note that $k_{h_0}+k_{h_0+1}$ is odd, so we integrate over $\theta_{h_0+1}$ and \textcircled{6} again gives us \eqref{equation tail 3}. 

Symmetrically, if we have $k_{j}$ is odd, then we can look at $h_1=\inf_{h\ge j}\{k_h {\rm \ is  \ even}\}$ and have $h_1\in [j+1,d-1]$. This in turn implies that $k_{h_1}+k_{h_1-1}$ is odd, so we integrate over $\theta_{h_1}$ and  use \textcircled{6}. We use the same argument in the following discussions.

Similarly if $i>0$, with $d(i,j)>k$ implying $j-i>k$ as well as $d+i-j>k$, we can also have 
$$
\prod_{h=0}^{d-1} \cos^{k_h}(\theta_h-\theta_{h+1})=\Pi[0:i-1]\cdot \Pi[i:j-1]\cdot\Pi[j:d-1]
$$
where
\begin{align*}
&\Pi[0:i-1]=\prod_{h=0}^{i-1} \cos^{k_h}(\theta_h-\theta_{h+1})\\
&\Pi[i:j-1]=\prod_{h=i}^{j-1} \cos^{k_h}(\theta_h-\theta_{h+1})\\
&\Pi[j:d-1]=\prod_{h=j}^{d-1} \cos^{k_h}(\theta_h-\theta_{h+1}). 
\end{align*}
And again note that 
\begin{align*}
&\cos(\theta_j-\theta_i)\left(\prod_{h=0}^{d-1} \cos^{k_h}(\theta_h-\theta_{h+1})\right)\\
&\hspace{0.2 in}=\cos(\theta_j-\theta_i)\cos^{k_{j-1}}(\theta_j-\theta_{j-1})\cos^{k_j}(\theta_j-\theta_{j+1})\left(\prod_{h\in \{0, \cdots, d-1\}\setminus\{j-1,j\}} \cos^{k_h}(\theta_h-\theta_{h+1})\right)
\end{align*}
and that 
\begin{align*}
&\cos(\theta_j-\theta_i)\left(\prod_{h=0}^{d-1} \cos^{k_h}(\theta_h-\theta_{h+1})\right)\\
&\hspace{0.2 in}=\cos(\theta_i-\theta_j)\cos^{k_{i-1}}(\theta_i-\theta_{i-1})\cos^{k_i}(\theta_i-\theta_{i+1})\left(\prod_{h\in \{0, \cdots, d-1\}\setminus\{i-1,i\}} \cos^{k_h}(\theta_h-\theta_{h+1})\right).
\end{align*}
So if either $k_{i-1}+k_{i}$ or $k_{j-1}+k_{j}$ is a even number,\textcircled{6} again gives us \eqref{equation tail 3}. 

Now suppose both $k_{i-1}+k_{i}$ and $k_{j-1}+k_{j}$ are odd. If either $k_i$ or $k_{j-1}$ is odd, we can without loss of generality assume the odd one is $k_{j-1}$. Note that 
$$
\sum_{h=i}^{j-1}k_h\le k<j-i. 
$$
Let $h_0=\sup_{h\le j-1}\{k_h {\rm \ is  \ even}\}$. Then $h_0\in[i,j-2]$. Then again we have that $k_{h_0}+k_{h_0+1}$  is odd, so we can integrate over $\theta_{h_0+1}$ and \textcircled{6} again gives us \eqref{equation tail 3}. 

Otherwise, we must have both $k_{i-1}$ and $k_j$ are odd numbers. Again note that 
$$
\sum_{h=0}^{i-1}k_h+ \sum_{h=j}^{d-1}k_h\le k<d+i-j.
$$
At least one of the $k_h$'s above must be 0, and let's say again without loss of generality it is in $[0,i-1]$. Once more let $h_0=\sup_{h\le i-1}\{k_h {\rm \ is  \ even}\}$. Then $h_0\in[0,i-2]$, and $k_{h_0}+k_{h_0+1}$ is odd so we can once again integrate over $\theta_{h_0+1}$ to use \textcircled{6} to gives us \eqref{equation tail 3}. Combining all the possible situations together, the proof of this lemma is complete. 
\end{proof}

With Lemma \ref{Lemma tail}, one can immediately see that for any $0\le i\le d-1$ and any $j$ such that $d(i,j)\ge 6$, 
$$
G_{d-1}^{(i,j)}=\mathcal{\hat E}^{(i,j)}_d=\left(\frac{1}{2\pi}\right)^{d-1}\int_{[-\pi,\pi]^{d-1}}\frac{\cos(\theta_j-\theta_i)\hat\phi^6_{d-1}(\theta)}{1-\hat\phi_{d-1}(\theta)}d\theta 
$$
which immediately implies that 
\beq
\label{equation sharp 4}
\left|G_{d-1}^{(i,j)}\right|\le \left(\frac{1}{2\pi}\right)^{d-1}\int_{[-\pi,\pi]^{d-1}}\frac{\hat\phi^6_{d-1}(\theta)}{1-\hat\phi_{d-1}(\theta)}d\theta.  
\eeq
Then recalling that in the proof Theorem \ref{Theorem diagonal} we have $\hat X_1,\hat X_2,\cdots, \hat X_{d-1}$ be i.i.d. uniform random variables on $[-\pi,\pi]$ and
$$
\hat Y_{d-1}=\frac{1}{d}\left(\cos(\hat X_1)+\sum_{i=1}^{d-2}\cos(\hat X_{i+1}-\hat X_{i})+\cos(\hat X_{d-1}) \right)\in [-1,1].
$$
And we define
$$
\bar Z_{d-1}=\left\{
\begin{aligned}
&\frac{\hat Y_{d-1}^6}{1-\hat Y_{d-1}}, \ \  \hat Y_{d-1}<1\\
&0, \hspace{0.65 in} \hat Y_{d-1}=1.
\end{aligned}
\right.
$$ 
Then again we have for any $i,j$
\beq
\label{expectation bar}
\mathcal{\hat E}^{(i,j)}_d\le E[\bar Z_{d-1}]. 
\eeq
Recall the event $\hat A_{d-1}=\{|\hat Y_{d-1}|\le d^{-0.4}\}$, then for any $d\ge 6$, 
$$
E[\bar Z_{d-1}]\le \frac{d^{-2.4}}{1-d^{-0.4}}P(\hat A_{d-1})+E[\bar Z_{d-1}\ind_{\hat A^c_{d-1}}]\le 2d^{-1.6}+E[\bar Z_{d-1}\ind_{\hat A^c_{d-1}}].
$$
Then recall
$$
\hat B_{d-1}=\left\{\sqrt{\hat X_1^2+\hat X_2^2+\cdots+\hat X_{d-1}^2}\le\frac{1}{d}\right\}.
$$
We can similarly have 
\beq
\label{bar expectation decomposition}
\begin{aligned}
E[\bar Z_{d-1}]&\le 2d^{-2.4}+E[\bar Z_{d-1}\ind_{\hat A^c_{d-1}\cap \hat B_{d-1}^c}]+E[\bar Z_{d-1}\ind_{A^c_{d-1}\cap \hat B_{d-1}}]\\
&\le 2d^{-2.4}+P(\hat A^c_{d-1})\max_{\omega\in \hat B_{d-1}^c} \{\bar Z_{d-1}(\omega)\}+E[\bar Z_{d-1}\ind_{\hat B_{d-1}}]. 
\end{aligned}
\eeq
Noting that in \eqref{hat term 3 1}-\eqref{hat term 2 2} and \eqref{hat term 2 3}, we find upper bounds for $\hat Z_{d-1}\ind_{\hat B_{d-1}}$ and $\max_{\omega\in \hat B_{d-1}^c} \{\bar Z_{d-1}(\omega)\}$ using $1/(1-\hat Y_{d-1})$ which is also an upper bound for the smaller corresponding terms with $\bar Z_{d-1}$. Thus \eqref{hat term 3} and \eqref{hat term 2} give us the second and third term in \eqref{bar expectation decomposition} is also $o(d^{-2.4})$. Which implies there is a $C_1<\infty$ such that for sufficiently large even number $d$, 
$$
\left|G_{d-1}^{(i,j)}\right|\le C_1d^{-2.4}
$$
whenever $d(i,j)\ge 6$, and that 
$$
\mathcal{\hat E}^{(i,j)}_d\le C_1d^{-2.4}
$$
for all $0\le i,j\le d-1$. Combining the observation here with Lemma \ref{Lemma tail}, we have for sufficiently large $d$ any $i$
\beq
\label{equation sharp 5}
\begin{aligned}
\sum_{j=0}^{d-1} G_{d-1}(e_j-e_i)-1&=\sum_{j=0}^{d-1}\mathcal{\hat E}^{(i,j)}_d+ \sum_{j: d(i,j)\le 5} \sum_{p=1}^5 \left(\frac{1}{2\pi}\right)^{d-1}\left( \int_{[-\pi,\pi]^{d-1}}\cos(\theta_j-\theta_i)\hat\phi^p_{d-1}(\theta) d\theta\right)\\
&\le C_1 d^{-1.4}+ \sum_{j: d(i,j)\le 5} \sum_{p=1}^5 \left(\frac{1}{2\pi}\right)^{d-1}\left( \int_{[-\pi,\pi]^{d-1}}\cos(\theta_j-\theta_i)\hat\phi^p_{d-1}(\theta) d\theta\right). 
\end{aligned}
\eeq
Note that for any $n>0$, $|\{j: d(i,j)=n\}|\le 2$. So the second term in \eqref{equation sharp 5} is just a finite summation of no more than 55 terms. When $p=1$, if $d(i,j)=0$,
$$
\left(\frac{1}{2\pi}\right)^{d-1}\int_{[-\pi,\pi]^{d-1}}\cos(\theta_j-\theta_i)\hat\phi_{d-1}(\theta) d\theta=\left(\frac{1}{2\pi}\right)^{d-1}\int_{[-\pi,\pi]^{d-1}}\hat\phi_{d-1}(\theta) d\theta=0.
$$
And if $d(i,j)=1$,
$$
\left(\frac{1}{2\pi}\right)^{d-1}\int_{[-\pi,\pi]^{d-1}}\cos(\theta_j-\theta_i)\hat\phi_{d-1}(\theta) d\theta=\frac{1}{4\pi^2d}\int_{[-\pi,\pi]^2}\cos^2(\theta_j-\theta_i)d\theta_jd\theta_i=\frac{1}{2d}.
$$
For $d(i,j)\ge2$, by Lemma \ref{Lemma tail}
$$
\left(\frac{1}{2\pi}\right)^{d-1}\int_{[-\pi,\pi]^{d-1}}\cos(\theta_j-\theta_i)\hat\phi_{d-1}(\theta) d\theta=0. 
$$
And for $p\ge 2$ and any $i,j$
{\small \begin{align*}
\left|\left(\frac{1}{2\pi}\right)^{d-1}\int_{[-\pi,\pi]^{d-1}}\cos(\theta_j-\theta_i)\hat\phi^p_{d-1}(\theta) d\theta\right|&=\left|\left(\frac{1}{2\pi}\right)^{d-1}\int_{[-\pi,\pi]^{d-1}}\hat\phi^2_{d-1}(\theta) \cos(\theta_j-\theta_i)\hat\phi^{p-2}_{d-1}(\theta) d\theta\right|\\
&\le \left(\frac{1}{2\pi}\right)^{d-1}\int_{[-\pi,\pi]^{d-1}}\hat\phi^2_{d-1}(\theta)d\theta=\frac{1}{2d}. 
\end{align*}}
Thus we have shown that all terms in this finite summation is either 0 or $O(d^{-1})$. Take $C=28$ and the proof of Theorem \ref{Theorem sharp} is complete. \qed
\begin{remark}
\label{remark  tedious}
It is clear that the upper bound $C=28$ we find here is not precise since here we only want the right order and are actually having very weak upper bounds for those 55 terms in the summation. Actually, any $C>3/2$ will be a good upper bound for sufficiently large $d$. Among the 55 terms in the summation, one can easily see that the term $j=i$, $p=2$ and the two terms with $d(i,j)=1$, $p=1$ are the only ones $\sim d^{-1}$ and each of them is $1/(2d)+o(d^{-1})$. All the other terms are either 0 or $o(d^{-1})$. The calculation is trivial calculus but very tedious, especially for someone who is reading (or writing) this not too short paper.   
\end{remark}

\section{~}\label{sec:appa}
In this appendix we prove that the monotonicity fails when considering covering probability with repetitions.

 {\bf Proof of  Proposition \ref{Counterexample 1}:} To show the first part of Equation \eqref{probability 1} and \eqref{probability 2}, note that $\left\{{\rm Trace}(\mathcal{P})\otimes N_\mathcal{P}\subseteq \{X_n\}_{n=0}^\infty\right\}$  is a subset of event $\{\tau_A<\infty\}$. Thus by strong Markov property and symmetry of simple random walk 
\beq
\begin{aligned}
\label{probability 11}
&P\left({\rm Trace}(\mathcal{P})\otimes N_\mathcal{P}\subseteq \{X_n\}_{n=0}^\infty\right)\\
& \ \ =P_w(\tau_y<\infty, \tau_z<\infty)P(\tau_w=\tau_A)+P_z(\tau_y<\infty, \tau_w<\infty)P(\tau_z=\tau_A)\\
& \ \ \ +P_y(\tau_z<\infty, \tau_w<\infty)P(\tau_y=\tau_A)\\
&  \ \ =2P_o(\tau_y=\tau_A)[P_o(\tau_y<\tau_\omega)+P_o(\tau_\omega<\tau_y)]P_o(\tau_y<\infty)\\
& \ \ \ +2P_o(\tau_\omega=\tau_A)P_o(\tau_y<\tau_z)P_o(\tau_w<\infty).
\end{aligned}
\eeq
Similarly, note that $\left\{{\rm Trace}(\mathcal{P}')\otimes N_{\mathcal{P}'}\subseteq \{X_n\}_{n=0}^\infty\right\}$ is a subset of $\{\tau_{A_1}<\infty\}$, where $A_1=\ZZ^3\setminus\{y,w\}$
\beq
\label{probability 12}
\begin{aligned}
&P\left({\rm Trace}(\mathcal{P}')\otimes N_{\mathcal{P}'}\subseteq \{X_n\}_{n=0}^\infty\right)\\
& \ \ =P_y(\tau_y<\infty, \tau_w<\infty)P(\tau_y=\tau_{A_1})+P_w(\tau_{2,y}<\infty)P(\tau_w=\tau_{A_1})\\
&  \ \ =P_o(\tau_y<\tau_\omega)[P_o(\tau_o<\tau_y)+P_o(\tau_y<\tau_o)]P_o(\tau_y<\infty) \\
&  \ \ \ +P_o(\tau_w<\tau_y)P_o(\tau_y<\infty)P_o(\tau_0<\infty)
\end{aligned}
\eeq	
where $\tau_{2,y}$ is defined in \eqref{n stopping time}. To calculate the probability we have in  \eqref{probability 11} and \eqref{probability 12}, one may first note that 
$$
P_o(\tau_y<\infty)=\frac{G(y)}{G(o)},  \ P_o(\tau_w<\infty)=\frac{G(w)}{G(o)},
$$
where $G(\cdot)$ is the Green's function of 3-dimensional simple random walk. I.e., 
$$
G(x)=\frac{1}{(2\pi)^3}\int_{[-\pi,\pi]^3}\frac{1}{1-\phi(\theta)} e^{-i y\cdot \theta}d\theta
$$
with 
$$
\phi(\theta)=\frac{1}{3}[\cos(\theta_1)+\cos(\theta_2)+\cos(\theta_3)]. 
$$
Thus
\beq
\label{G(o)}
G(o)=\frac{1}{(2\pi)^3}\int_{[-\pi,\pi]^3}\frac{1}{1-\frac{1}{3}[\cos(\theta_1)+\cos(\theta_2)+\cos(\theta_3)]} d\theta\approx 1.5153,
\eeq
\beq
\label{G(y)}
G(y)=\frac{1}{(2\pi)^3}\int_{[-\pi,\pi]^3}\frac{\cos(\theta_1)}{1-\frac{1}{3}[\cos(\theta_1)+\cos(\theta_2)+\cos(\theta_3)]} d\theta \approx 0.5153,
\eeq
\beq
\label{G(2y)}
G(2y)=\frac{1}{(2\pi)^3}\int_{[-\pi,\pi]^3}\frac{\cos(2\theta_1)}{1-\frac{1}{3}[\cos(\theta_1)+\cos(\theta_2)+\cos(\theta_3)]} d\theta \approx 0.2563 ,
\eeq
\beq
\label{G(w)}
G(w)=\frac{1}{(2\pi)^3}\int_{[-\pi,\pi]^3}\frac{\cos(\theta_1+\theta_2)}{1-\frac{1}{3}[\cos(\theta_1)+\cos(\theta_2)+\cos(\theta_3)]} d\theta \approx 0.3301,
\eeq
\beq
\label{tau_y}
P_o(\tau_y<\infty)=\frac{\int_{[-\pi,\pi]^3}\frac{\cos(\theta_1)}{1-\frac{1}{3}[\cos(\theta_1)+\cos(\theta_2)+\cos(\theta_3)]} d\theta}{\int_{[-\pi,\pi]^3}\frac{1}{1-\frac{1}{3}[\cos(\theta_1)+\cos(\theta_2)+\cos(\theta_3)]} d\theta}\approx 0.3401,
\eeq
\beq
P_o(\tau_o<\infty)=P_o(\tau_y<\infty)\approx 0.3401,
\eeq
\beq
\label{tau_2y}
P_o(\tau_{2y}<\infty)=\frac{\int_{[-\pi,\pi]^3}\frac{\cos(2\theta_1)}{1-\frac{1}{3}[\cos(\theta_1)+\cos(\theta_2)+\cos(\theta_3)]} d\theta}{\int_{[-\pi,\pi]^3}\frac{1}{1-\frac{1}{3}[\cos(\theta_1)+\cos(\theta_2)+\cos(\theta_3)]} d\theta}\approx 0.1691,
\eeq
and
\beq
\label{tau_w}
P_o(\tau_w<\infty)=\frac{\int_{[-\pi,\pi]^3}\frac{\cos(\theta_1+\theta_2)}{1-\frac{1}{3}[\cos(\theta_1)+\cos(\theta_2)+\cos(\theta_3)]} d\theta}{\int_{[-\pi,\pi]^3}\frac{1}{1-\frac{1}{3}[\cos(\theta_1)+\cos(\theta_2)+\cos(\theta_3)]} d\theta}\approx 0.2178. 
\eeq
Then for $A_2=\ZZ^3\setminus \{z\}$, we have
$$
P_o(\tau_y<\tau_z)=P_o(\tau_y<\tau_{A_2})=\frac{G_{A_2}(o,y)}{G_{A_2}(y,y)}
$$
where $G_{A_2}(\cdot)$ is the Green's function for set $A_2$, see Section 4.6 of \cite{randomwalkbook} for reference. Then by Proposition 4.6.2 of \cite{randomwalkbook}, 
$$
G_{A_2}(o,y)=G(y)-P_o(\tau_y<\infty)G(w), \ G_{A_2}(y,y)=G(o)-P_o(\tau_w<\infty)G(w),
$$
which gives 
\beq
\label{tau_y<tau_z}
P_o(\tau_y<\tau_z)=P_o(\tau_z<\tau_y)=\frac{G(y)-P_o(\tau_y<\infty)G(w)}{G(o)-P_o(\tau_w<\infty)G(w)}\approx 0.2792. 
\eeq
Similarly, for $A_3=\ZZ^3\setminus \{y,z\}$ we have 
$$
P_o(\tau_w=\tau_A)=P_o(\tau_w<\tau_{A_3})=\frac{G_{A_3}(o,w)}{G_{A_3}(w,w)},
$$
where
$$
G_{A_3}(o,w)=G(w)-[P_o(\tau_y<\tau_z)+P_o(\tau_z<\tau_y)]G(y)
$$
and
$$
G_{A_3}(w,w)=G(o)-[P_o(\tau_y<\tau_z)+P_o(\tau_z<\tau_y)]G(y)
$$
which gives 
\beq
\label{tau_w=tau_A}
P_o(\tau_w=\tau_A)=\frac{G(w)-2P_o(\tau_y<\tau_z)G(y)}{G(o)-2P_o(\tau_y<\tau_z)G(y)}\approx0.0344. 
\eeq
Then for $A_4=\ZZ^3\setminus \{w\}$, 
$$
P_o(\tau_y<\tau_w)=P_o(\tau_y<\tau_{A_4})=\frac{G_{A_4}(o,y)}{G_{A_4}(y,y)}
$$
where
$$
G_{A_4}(o,y)=G(y)-P_o(\tau_w<\infty)G(y), \ G_{A_4}(y,y)=G(o)-P_o(\tau_y<\infty)G(y).
$$
Thus
\beq
\label{tau_y<tau_w}
P_o(\tau_y<\tau_w)=P_o(\tau_z<\tau_w)=\frac{G(y)-P_o(\tau_w<\infty)G(y)}{G(o)-P_o(\tau_y<\infty)G(y)}\approx 0.3008. 
\eeq
And for $A_5=\ZZ^3\setminus \{y\}$, 
$$
P_o(\tau_w<\tau_y)=P_o(\tau_w<\tau_{A_5})=\frac{G_{A_5}(o,w)}{G_{A_5}(w,w)},
$$
where
$$
G_{A_5}(o,w)=G(w)-P_o(\tau_y<\infty)G(y), \ G_{A_4}(w,w)=G(o)-P_o(\tau_y<\infty)G(y).
$$
Thus
\beq
\label{tau_w<tau_y}
P_o(\tau_w<\tau_y)=P_o(\tau_w<\tau_z)=\frac{G(w)-P_o(\tau_y<\infty)G(y)}{G(o)-P_o(\tau_y<\infty)G(y)}\approx 0.1155. 
\eeq
And for $A_6=\ZZ^3\setminus \{z,w\}$,
$$
P_o(\tau_y=\tau_A)=P_o(\tau_y<\tau_{A_6})=\frac{G_{A_6}(o,y)}{G_{A_6}(y,y)},
$$
where
$$
G_{A_6}(o,y)=G(y)-P_o(\tau_w<\tau_y)G(y)-P_o(\tau_y<\tau_w)G(w)
$$
and
$$
G_{A_6}(y,y)=G(o)-P_o(\tau_w<\tau_y)G(w)-P_o(\tau_y<\tau_w)G(y).
$$
Thus we have
\beq
\label{tau_y=tau_A}
P_o(\tau_y=\tau_A)=\frac{G(y)-P_o(\tau_w<\tau_y)G(y)-P_o(\tau_y<\tau_w)G(w)}{G(o)-P_o(\tau_w<\tau_y)G(w)-P_o(\tau_y<\tau_w)G(y)}\approx 0.2696
\eeq
which by symmetry also equals to $P_o(\tau_z=\tau_A)$. Finally for $P_o(\tau_o<\tau_y)$ and $P_o(\tau_o<\tau_y)$, using one step argument at time 0,
$$
P_o(\tau_o<\tau_y)=\frac{2}{3}P_o(\tau_y<\tau_w)+\frac{1}{6} P_o(\tau_y<\tau_{2y})
$$
and
$$
P_o(\tau_y<\tau_o)=\frac{1}{6}+\frac{2}{3}P_o(\tau_w<\tau_y)+\frac{1}{6} P_o(\tau_{2y}<\tau_{y}).
$$
So again for $A_7=\ZZ^3\setminus \{2y\}$, we have
$$
P_o(\tau_y<\tau_{2y})=\frac{G_{A_7}(o,y)}{G_{A_7}(y,y)},
$$
where
$$
G_{A_7}(o,y)=G(y)-P_o(\tau_{2y}<\infty)G(y)
$$
and 
$$
G_{A_7}(y,y)=G(o)-P_o(\tau_{y}<\infty)G(y).
$$
Thus
\beq
\label{tau_o<tau_y}
P_o(\tau_o<\tau_y)=\frac{2}{3}P_o(\tau_y<\tau_w)+\frac{1}{6}\frac{G(y)-P_o(\tau_{2y}<\infty)G(y)}{G(o)-P_o(\tau_{y}<\infty)G(y)}\approx 0.2538. 
\eeq 
And for $P_o(\tau_{2y}<\tau_{y})$, recalling that $A_5=\ZZ^3\setminus \{y\}$ we have
$$
P_o(\tau_{2y}<\tau_{y})=\frac{G_{A_5}(o,2y)}{G_{A_5}(2y,2y)},
$$
where
$$
G_{A_5}(o,2y)=G(2y)-P_o(\tau_y<\infty)G(y), \ G_{A_5}(2y,2y)=G(o)-P_o(\tau_y<\infty)G(y).
$$
Thus 
\beq
\label{tau_y<tau_o}
P_o(\tau_y<\tau_o)=\frac{1}{6}+\frac{2}{3}P_o(\tau_w<\tau_y)+\frac{1}{6} \frac{G(2y)-P_o(\tau_y<\infty)G(y)}{G(o)-P_o(\tau_y<\infty)G(y)}\approx 0.2538.
\eeq 
At this point, we finally have all the variables needed calculated, apply (\ref{G(o)}-\ref{tau_y<tau_o}) to \eqref{probability 11} and \eqref{probability 12}, the proof of  Proposition \ref{Counterexample 1} is complete.

\end{document}